\documentclass{amsart}
\usepackage{amsfonts,amssymb,amsmath,amsthm}
\usepackage{url}
\usepackage{enumerate}
\def\refname{R\'ef\'erences}
\def\proofname{D\'emonstration}

 \usepackage{hyperref}
\usepackage{fancyhdr}
\newtheorem{theorem}{Th\'eor\`eme}[section]
\newtheorem{lemma}[theorem]{Lemme}

\newtheorem{proposition}[theorem]{Proposition}
\newtheorem{corollary}[theorem]{Corollaire}
\newtheorem{example}[theorem]{Exemple}
\theoremstyle{definition}

\newtheorem{heuristic}[theorem]{Heuristique}
\newtheorem{remark}[theorem]{Remarque}
\numberwithin{equation}{section}
\def\notdiv{\nmid}
\def\div{\,\vert\,}
\def\N{\mathbb{N}}

\def\C{\mathbb{C}}
\def\Q{\mathbb{Q}}
\def\Z{\mathbb{Z}}

\def\Sauf{\!\setminus\!}
\def\No{{\rm N}}

\def\Frac#1#2{\hbox{\footnotesize $\displaystyle \frac{#1}{#2}$}}
\def\plus{\displaystyle\mathop{\raise 2.0pt \hbox{$\bigoplus $}}\limits}
\def\prd{ \displaystyle\mathop{\raise 2.0pt \hbox{$\prod$}}\limits}
\def\sm{  \displaystyle\mathop{\raise 2.0pt \hbox{$\sum$}}\limits}

\author{Georges  Gras}
\address{Villa la Gardette \\ chemin Ch\^ateau Gagni\`ere \\  F--38520 Le Bourg d'Oisans.}
\email{g.mn.gras@wanadoo.fr 
 {\it url : }\url{https://www.researchgate.net/profile/Georges_Gras/?dbw=true} \, -- \, \url{http://monsite.orange.fr/maths.g.mn.gras/}}



\begin{document}
 
\title[Etude probabiliste des $p$-quotients de Fermat]{Etude probabiliste des $p$-quotients de Fermat}

\begin{abstract} 
For a fixed integer $a\geq 2$, we suggest that the probability of nullity of the Fermat quotient $q_p(a)$ is much lower than~$\frac{1}{p}$
for any arbitrary large prime number~$p$. For this we use various heuristics, justified by means of numerical computations and analytical results, which may imply the finiteness of the $q_p(a)$ equal to 0 and the existence of integers $a$ such that $q_p(a)\ne 0 \ \forall p$.
However no proofs are obtained concerning these heuristics.
\end{abstract}

\date{4 Septembre 2014}

\maketitle

\vspace{-0.8cm}

\section{Introduction}\label{intro}
 Nous \'etudions la probabilit\'e de nullit\'e du $p$-quotient de Fermat $q_p(a)$, de $a$ fix\'e dans $\N\Sauf \{0, 1\}$, 
$p$ \'etant la variable,
\`a partir du fait que ceci a lieu si et seulement si $p^2$ divise la valeur en $a$ du $m$-i\`eme polyn\^ome cyclotomique $\Phi_m$, o\`u $m\div p-1$ est l'ordre de $a$ modulo $p$ (par abus $q_p(a) = 0$ signifie $\frac{a^{p-1} - 1}{p} \equiv 0 \pmod p$).

\smallskip\noindent
Dans un premier temps, nous utilisons un r\'esultat g\'en\'eral de Andrew Granville (1998) qui, sous la v\'eracit\'e de la conjecture $ABC$, 
permet, gr\^ace \`a un principe local-global diophantien, de d\'eterminer (pour $f \in \Z[x]$) la densit\'e
des entiers $A\in\N$ tels que $f(A)$ est sans facteur carr\'e. 
Pour $\Phi_m$, la densit\'e relative \`a la seule condition locale $p^2 \notdiv \Phi_m(A)$,
pour $p\equiv 1\! \pmod m$, est  \'egale \`a $1 - \frac{\varphi(m)}{p^2}$ o\`u $\varphi$ est l'indicateur d'Euler, celle relative 
\`a la condition $\Phi_m(A)$ sans facteur carr\'e \'etant  \'egale au produit $\prod_{p \equiv 1 \!\!\pmod m}(1 - \frac{\varphi(m)}{p^2})$ des densit\'es locales.
Notons que pour tout $p$, la {\it densit\'e} des $A \in \N \Sauf p\N$ tels que $q_p(A)=0$ est trivialement $\frac{1}{p}$
(resp. $\frac{p-1}{p^2}$ pour celle des $A\in \N$).

\smallskip\noindent
On en d\'eduit l'heuristique sui\-vante reposant sur le fait que les proba\-bilit\'es sont inf\'erieures aux densit\'es  cor\-respondantes
(i.e., lorsque $a$ est remplac\'e par la variable al\'eatoire $A \in \N$)~: pour $a$ fix\'e et $p$ arbitraire assez grand, on a la majoration~:

\medskip
\centerline {${\rm Prob}\big( q_p(a)=0\big) < \Frac{1}{p\,(p-1)^2} \sm_{d \div p-1} \varphi(d)^2\!<\Frac{1}{p}$, }

\smallskip
qui ne renseigne que partiellement sur la finitude ou non des $q_p(a)$ nuls. 

\smallskip\noindent
Dans un second temps, nous montrons comment tenir compte d'avantage du fait qu'en pratique $a$ est fix\'e une fois pour toutes 
et que si $q_p(a) = 0$ alors $q_p(a^j) = 0$ pour les exposants $j$ tels que $a^j \in [2, p[$
($p$ \'etant la variable al\'eatoire tendant vers l'infini). 
On \'etudie alors une heuristique stipulant l'existence d'une loi de proba\-bilit\'e binomiale, pour le nombre d'entiers $z \in [2, p[$ 
tels que $q_p(z)=0$, \`a savoir ${\rm Prob}\big(\big\vert \big\{z \in [2,p[,\, q_p(z)=0 \big\} \big\vert  \geq n\big) = 
1 - \sm_{j=0}^{n-1} \hbox{$\binom{p-2} {j}$}\Frac{1}{p^j}\Big (1-\Frac{1}{p} \Big)^{p-2-j}$\!\!,
qui implique\-rait, via le principe de Borel--Cantelli, la finitude des $p$ tels que $q_p(a)=0$.

\smallskip\noindent
Enfin, en utilisant le fait que le produit formel $\widetilde {\mathcal P}(A) =\prd_{m \geq 1}\Frac{\Phi_m(A)}{{\rm p.g.c.d.\,}(\Phi_m(A),m)}$ 
est divisible par tous les nombres premiers et que $q_p(A)=0$ si et seulement si $p^2 \div \widetilde {\mathcal P}(A)$, on obtient la densit\'e 
des $A\in\N$ tels que $q_p(A) \ne 0\ \forall p \leq x$ (cf.  Th\'eor\`eme \ref{theo2}).

\smallskip\noindent
En toute hypoth\`ese, on peut envisager que la probabilit\'e de nullit\'e de $q_p(a)$ (pour $a$ fix\'e et $p\to \infty$) est 
strictement inf\'erieure \`a $\frac{1}{p}$ et que la conjecture sur la finitude des premiers $p$ tels que $q_p(a)=0$ reste cr\'edible 
(conjecture qui est un cas parti\-culier des conjectures analogues que nous avons formul\'ees dans le cadre g\'en\'eral des r\'egulateurs 
$p$-adiques d'un nombre alg\'ebrique, cf. \cite{Gr}).

\section{Cyclotomie et quotients de Fermat} \label{section2}

\subsection{Rappels sur le quotient de Fermat}\label{subfermat}
Soit  $a \in \N \Sauf \{0, 1\}$ fix\'e. Soit $p$ un nombre premier ne divisant pas $a$. 
Soit $m = o_p(a)$, divisant $p-1$, l'ordre de $a$ modulo $p$ et soit $\xi$ une racine primitive 
$m$-i\`eme de l'unit\'e dans $\C$~; alors on peut \'ecrire $a^m-1 = \prod_{j=1}^m (a- \xi^j) \equiv 0 \pmod p$. 

\smallskip
Comme $m$ est l'ordre de $a$ modulo $p$, c'est le facteur de $a^m-1$ d\'efini par~:
$$\Phi_m(a) = \prd_{t \in (\Z/m\Z)^\times} (a- \xi^t)$$

qui est dans $p\,\Z$,  o\`u $\Phi_m$ est le $m$-i\`eme polyn\^ome cyclotomique.
De fa\c con pr\'ecise on a la relation $\Frac{a^{m} - 1}{p}  = 
\Frac{\Phi_m(a)}{p} \times \prd_{\substack {d\vert m, \\ d \ne m}} \Phi_d(a)$, o\`u
$\prd_{\substack {d\vert m, \\d \ne m}}  \Phi_d(a) \not\equiv 0 \!\!\pmod p$~;
en effet, si l'on avait $p\div \Phi_d(a)$ pour $d\div m, d\ne m$, alors on aurait $p\div a^d-1$ et $m$ ne serait pas 
l'ordre de $a$ modulo $p$. On a donc l'implication $m = o_p(a) \Longrightarrow p\div \Phi_m(a)$.

\smallskip
 La r\'eciproque est inexacte~; par exemple, si $p=3$, $m=6$, 
$a=5$, on a $\Phi_m(a)=7\times p$ avec pour ordre de $a$ modulo $p$, $o_{p}(a)=2$ et $\Phi_2(a)=2\times p$ comme attendu,
mais on a ici $m = p\, .\, o_p(a)$ (i.e., ${\rm p.g.c.d.\,}(\Phi_m(a), m) = p$). 
Ce ph\'enom\`ene sera pr\'ecis\'e par le Th\'eor\`eme \ref{theorem}

\begin{remark}\label{rema1}
Si l'on pose $q_p(a) := \frac{a^{p-1}-1}{p}$, $q'_p(a) := \frac{a^{o_p(a)}-1}{p}$ et $p-1= t\,o_p(a)$, il vient
$q_p(a) \equiv t\, q'_p(a) \equiv \frac{-1}{o_p(a)} q'_p(a) \pmod p$~; on peut aussi envisager l'expression
$q''_p(a) := \frac{\Phi_{o_p(a)}(a) }{p}$. Ces diff\'erentes d\'efinitions possibles 
du quotient de Fermat sont \'equivalentes en ce qui concerne sa nullit\'e modulo $p$.

\smallskip
En particulier, on a $q_p(a) \equiv 0 \pmod p$ si et seulement si $\Phi_{o_p(a)}(a)\equiv 0 \pmod {p^2}$
 (pour diverses propri\'et\'es des quotients de Fermat on peut se reporter \`a \cite{EM}, \cite{Hat}, \cite{He}, \cite{KJ}, \cite{Sh}, \cite{OS},
 ainsi qu'\`a \cite{Si}, \cite{GM}, \cite{W} pour les liens avec la conjecture $ABC$). 
\end{remark}

\subsection{Utilisation des corps cyclotomiques}\label{subcyclo} Nous n'utilisons que des propri\'et\'es classiques
que l'on peut trouver dans \cite{Wa}.

\begin{lemma}\label{lem2} Soient $a \in \N \Sauf \{0, 1\}$, $p\notdiv a$, et $m \geq 1$. Alors la congruence
$\Phi_m(a) \equiv 0\!\pmod {p^h}$, $h\geq 1$, est \'equivalente \`a l'existence d'un  couple $(\xi, {\mathfrak P})$, unique \`a conju\-gaison pr\`es, tel que $a \equiv \xi \pmod { {\mathfrak P}^h}$, o\`u $\xi$ est une racine primitive $m$-i\`eme de l'unit\'e et ${\mathfrak P}$ 
un id\'eal premier de $\Q(\xi)$ au-dessus de $p$,  de degr\'e r\'esiduel $1$. 

 En outre, lorsque ceci a lieu, $m$ est n\'ecessairement de la forme $p^e .\, o_p(a)$, $e\geq 0$.
\end{lemma}

\begin{proof} La relation $a \equiv \xi \pmod { {\mathfrak P}^h}$, $h\geq 1$,
 prouve que ${\mathfrak P}$ est de degr\'e r\'esi\-duel~1 car
 l'anneau des entiers de $\Q(\xi)$ est $\Z[\xi]$ et $\xi$ est congrue \`a un rationnel modulo~${\mathfrak P}$.
 Un sens est donc \'evident puisque $\Phi_m(a) = \No_{\Q(\xi)/\Q} (a- \xi)$.

Supposons $\Phi_m(a) \equiv 0 \pmod {p^h}$, $h\geq 1$. Comme $\Phi_m(a) =\prd_{t \in (\Z/m\Z)^\times} (a- \xi^t) \equiv 0 \pmod {p^h}$, il existe ${\mathfrak P}_1 \div p$ dans $\Q(\xi)$ tel que $a- \xi  \equiv 0 \pmod {{\mathfrak P}_1}$.

\smallskip
Supposons que l'on ait $a- \xi  \equiv 0 \pmod {{\mathfrak P}_2}$, ${\mathfrak P}_2 \div p$, avec 
${\mathfrak P}_2 \ne {\mathfrak P}_1$~; il existe donc une conju\-guaison non triviale $\xi \mapsto \xi^t \ne \xi$ telle que
${\mathfrak P}_2 ={\mathfrak P}_1^{t^{-1}} \!\ne {\mathfrak P}_1$ et on obtient
$a- \xi^t  \equiv 0 \pmod {{\mathfrak P}_1}$, ce qui conduit \`a $\xi^t - \xi  \equiv 0 \pmod {{\mathfrak P}_1}$.
D'o\`u deux cas~:

\medskip
(i) $p\notdiv m\ \&\ \xi^t \ne \xi$~; alors $\xi^t - \xi $ est une unit\'e en $p$ (absurde).

\smallskip
(ii) $p \div m\ \&\ \xi^t \ne \xi$.

\smallskip
Donc si $p\notdiv m$, un seul id\'eal premier ${\mathfrak P}\div p$ intervient et on a $a- \xi  \equiv 0 \pmod{ {\mathfrak P}^h}$.

\smallskip
Examinons le cas $p \div m\ \&\ \xi^t  \ne  \xi$ en consid\'erant le sch\'ema suivant~:
\unitlength=0.6cm
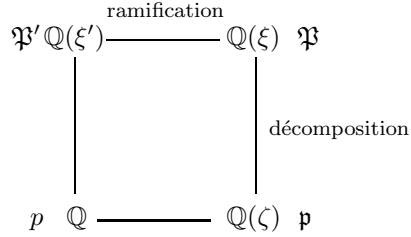
$$\vbox{\hbox{\hspace{-2cm}  \begin{picture}(11.5,5.5)
\put(4.2,4.50){\line(1,0){2.5}}
\put(4.0,0.50){\line(1,0){2.5}}
\put(3.50,1.1){\line(0,1){3.0}}
\put(7.50,1.1){\line(0,1){3.0}}

\put(6.85,0.4){$\Q(\zeta)\ \ {\mathfrak p}$}
\put(6.85,4.4){$\Q(\xi)\ \ {\mathfrak P}$}
\put(2.8,4.4){$\Q(\xi')$}
\put(2.1,4.4){${\mathfrak P}'$}
\put(2.5,0.4){$p$}
\put(3.3,0.40){$\Q$}

\put(7.8,2.4){{\footnotesize d\'ecomposition}}
\put(4.2,5.0){\footnotesize ramification}
\end{picture}   }} $$
\unitlength=1.0cm

 Si l'on pose $m= p^e m'$, $e\geq 1$, $p\notdiv m'$, et $\xi = \zeta\,\xi'$
($\zeta$ d'ordre $p^e$, $\xi'$ d'ordre $m'$), il vient $\zeta^t\xi'{}^t - \zeta\,\xi' \equiv 0 \pmod{ {\mathfrak P}_1}$.
Or on a toujours $\zeta \equiv 1 \pmod {{\mathfrak P}_1}$ car dans $\Q(\zeta)$ il y a un unique id\'eal premier
${\mathfrak p} = (1-\zeta)$ totalement ramifi\'e dans $\Q(\zeta)/\Q$, donc tel que ${\mathfrak P}_1 \div {\mathfrak p}$
et ${\mathfrak P}_2 \div {\mathfrak p}$
(si $p^e=2$, $\Q(\zeta)= \Q$ et ${\mathfrak p}=(2)$).

\smallskip
D'o\`u $\xi'{}^t -\xi'  \equiv 0 \pmod {{\mathfrak P}'_1 = {\mathfrak P}_1 \cap \Z[\xi'] }$ dans $\Q(\xi')$, 
et par cons\'equent $\xi'{}^t =\xi'$ (i.e., $t\equiv 1 \pmod {m'}$) puisque $p \notdiv m'$. 
Mais ceci implique ${\mathfrak P}_2 ={\mathfrak P}_1$ car $\Q(\xi)/ \Q(\xi')$ est totalement ramifi\'ee en $p$
et $t$ fixe $\Q(\xi')$  (absurde). 

\smallskip
On a donc obtenu dans tous les cas $a- \xi  \equiv 0 \pmod{ {\mathfrak P}^h}$ pour un unique ${\mathfrak P} \div p$.

\smallskip
Montrons enfin que $m' = o_p(a)$ dans tous les cas.
On a \`a ce stade $m=p^e\,m'$, $e\geq 0$, et $a\equiv \xi'  \pmod{ {\mathfrak P}' =  {\mathfrak P}\cap \Z[\xi']}$
puisque $\zeta \equiv 1 \pmod { {\mathfrak P}}$ (y compris si $e=0$ o\`u $\zeta=1$), ce qui implique 
$a^d \equiv 1 \pmod p$ (i.e., $\xi'{}^d  \equiv 1\pmod{ {\mathfrak P}'}$) si et seulement si $\xi'{}^d = 1$, 
d'o\`u $d \equiv 0 \pmod {m'}$~; d'o\`u le lemme.
\end{proof} 

Revenons \`a l'aspect r\'eciproque de l'implication $m = o_p(a) \Longrightarrow p\div \Phi_m(a)$ en te\-nant compte des questions de divisibilit\'es par $p^h$. D'apr\`es le lemme pr\'ec\'edent, si $p\div \Phi_m(a)$, on a $m = p^e m'$, $e\geq 0$, o\`u $m' = o_p(a)$,
et par cons\'equent $p \div \Phi_{m'}(a)$. 

\smallskip
Le cas $p\notdiv m$ est donc r\'esolu et conduit \`a l'\'equivalence partielle~:
$$p\div \Phi_m(a) \ \& \  p\notdiv m \  \Longleftrightarrow \  m = o_p(a). $$

Dans ce cas toute puissance $p^h$, $h\geq 1$, peut diviser $\Phi_m(a)$ (c'est le probl\`eme du quotient de Fermat pour $h\geq 2$).

\begin{lemma}\label{lem3} Supposons que pour $h\geq 1$, $p^h\div \Phi_m(a)$ avec $m = p^e m'$, $e\geq 1$, $p\notdiv m'$.
Alors n\'ecessairement $h=1$ (i.e., $\Phi_m(a) \not\equiv 0 \pmod {p^2}$) sauf si $p^e=m=2$,  auquel cas si 
$a=-1+ 2^h\,u$, $h \geq 1$ quelconque, on a $\Phi_2(a) =  2^h\ u$, $\Phi_1(a) = -2+ 2^h\, u$.
\end{lemma}

\begin{proof} 
On a donc par hypoth\`ese, d'apr\`es le Lemme \ref{lem2}, $a \equiv \xi  \pmod {{\mathfrak P}^h}$, pour $\xi = \zeta \xi'$ d'ordre $p^e m'$ ($\zeta$ d'ordre $p^e$, $\xi'$ d'ordre $m'$), et $a\equiv  \xi'  \pmod {{\mathfrak P}'{}^{h'}}$, ${\mathfrak P}' = {\mathfrak P}
\cap \Z[\xi']$, avec $h' \geq 1$ puisque $\zeta \equiv 1 \pmod{{\mathfrak P}}$~; on a l'identit\'e~:

\smallskip
\centerline{$a-\xi = a-\xi' + \xi'\,(1-\zeta),$}

\smallskip
 o\`u les ${\mathfrak P}$-valuations des termes sont respectivement $h,\  h' p^{e-1}(p-1), 1$.

\medskip
Si $h' p^{e-1}(p-1)>1$ on a n\'ecessairement $h=1$. Le cas $h' p^{e-1}(p-1)=1$ correspond au cas $p=2$, $h'=e=1$, donc
$o_2(a) = 1$, $\xi' = 1$, $\xi = -1$, $\Phi_2(a) = a+1$ et ${\rm p.g.c.d.}\,(2,\Phi_2(a))=2$ 
(e.g. $p=2$, $a = 23$, $m=2$, $\Phi_2(a) = 8\times 3$, $\Phi_1(a) = 2\times 11$, $h'=1$, $h=3$).
En dehors du cas $m=2$, $p=2$, $e=1$, on a $h=1$. 
\end{proof} 

En particulier, pour $m = p^e m'\ne 2$, $e \geq 1$, on a $p \div \Phi_m(a)$ et $p^2 \notdiv \Phi_m(a)$ (on rappelle que $m'=o_p(a)$).
Autrement dit, dans tous les cas o\`u $e \geq 1$, la valeur de $\Phi_m(a)$ ne peut renseigner sur le quotient de Fermat (dans le cas particulier
$p=2$, $m=2$, $\Phi_2(a) = a+1$, mais $q_2(a)=0$ signifie $a\equiv 1 \pmod 4$, or $a+1 \equiv 2 \pmod 4$).

\begin{theorem}\label{theorem}
 Pour tout $m \geq 1$, le p.g.c.d.  de $\Phi_m(a)$ et de $m$ est \'egal \`a 1 ou \`a
un nombre premier $p$. Dans ce dernier cas, $m= p^e .\, o_p(a)$, $e\geq 1$. R\'eciproquement,
pour tout premier $p$ et tout $e\geq 1$, $m= p^e .\, o_p(a)$ conduit \`a ${\rm p.g.c.d.\,}(\Phi_m(a), m) = p$.

Autrement dit, on a l'\'equivalence (pour tout $p$ et tout $m$)~:
$$p\div \Phi_m(a) \  \Longleftrightarrow \  m = p^e .\, o_p(a), \ e\geq 0, $$
\end{theorem}

\begin{proof} Si $p$ et $q$, $p\ne q$, sont des nombres premiers divisant $m$ et $\Phi_m(a)$, on a n\'ecessairement
$m = p^e q^f m''$, $e, f \geq 1$, avec $o_p(a) = q^f m'' \div p-1$ et $o_q(a) = p^e m'' \div q-1$, qui suppose $q<p$ et $p<q$ (absurde). 

\smallskip
Enfin montrons que tout $p$ premier et $e\geq 1$ conviennent pour $m=p^e .\, o_p(a)$. Comme $p\div \Phi_{o_p(a)}(a)$, on a
$a \equiv \xi' \pmod {{\mathfrak P}'}$ dans $\Q(\xi')$ ($\xi'$ d'ordre $o_p(a)$)~; donc pour toute racine $\zeta$ d'ordre $p^e$, 
et pour ${\mathfrak P} \div {\mathfrak P}'$ dans $\Q(\zeta\xi')$, on a $a \equiv \zeta\xi' \pmod {\mathfrak P}$
(d'o\`u le r\'esultat par le Lemme \ref{lem2}). Il est clair que ${\rm p.g.c.d.}\,(m,\Phi_m(a))=p$.
\end{proof} 

Nous r\'eserverons la notation $r$ au cas o\`u $m = r^e  .\, o_r(a)$, $e\geq 1$, car $r$ n'intervient pas 
pour le calcul des $p$-quotients de Fermat de $a$ pour $p \div \Phi_m(a)$. Autrement dit la consid\'eration de $p$
signifiera $p \div \Phi_m(a)$, $p \notdiv m$ (\'equivalent \`a $p\ne r$ si $m$ est de la forme pr\'ec\'edente avec $e\geq 1$).

\subsection{D\'efinition des nombres $\widetilde \Phi_m(a)$, $m\geq 1$}\label{subtild}
 On peut donc consid\'erer dans tous les cas $\widetilde \Phi_m(a) := \Frac{ \Phi_m(a)}{{\rm p.g.c.d.}\,( \Phi_m(a),m )}\  \ 
 \hbox{qui est \'egal \`a}\  \, \Phi_m (a) \ \,  \hbox{\rm ou \`a}\ \  \Frac{\Phi_{r^e .\, o_r(a)}(a)}{r}, \  e\geq 1$,

pour \'eliminer le facteur premier $r$ \'eventuel (ramifi\'e dans $\Q(\xi)/\Q$).
Dans le second cas $m = r^e  .\, o_r(a)$, $ e\geq 1$, si $p\ne r$ divise $\Phi_{m}(a)$, alors $m = o_p(a)$ et on a
$p\equiv 1 \pmod {r^e  .\, o_r(a)}$. 

\smallskip
Dans le cas o\`u ${\rm p.g.c.d.\,}( \Phi_m(a),m) = r$, la nullit\'e du
$r$-quotient de Fermat de $a$ est donn\'ee via $\Frac{\Phi_{o_r(a)}(a)}{r}$ en g\'en\'eral distinct des 
$\Frac{\Phi_{r^e .\, o_r(a)}(a)}{r}$ pour $e\geq 1$ puisque dans ce cas, et pour $r^e .\, o_r(a) \ne 2$, 
$\Phi_{r^e .\, o_r(a)}(a) \not\equiv 0 \pmod {r^2}$ (cf. Lemme \ref{lem3}).

Par exemple, pour $r=29$ et $a=14$ on a $o_{29}(a) = 28$, $\Frac{\Phi_{ {29} \, .\, 28}(a)}{29} =  F \not\equiv 0 \pmod {29}$ 
mais $\Frac{\Phi_{28}(a)}{29} = 29 \times F'$ (i.e., $q_{29}(14) = 0$).

Pour $m=2$ et $a$ impair, on a $r=2$ et $\widetilde \Phi_2(a) = \Frac{a+1}{2}$ qui peut
\^etre divisible par une puissance de 2 arbitraire contrairement au cas g\'en\'eral (cf. Lemme \ref{lem3}).

\subsection{D\'ecomposition en facteurs premiers de $\widetilde\Phi_m(a)$}\label{subdfp}

Soit $m \ne 2$~; d'apr\`es les r\'esultats pr\'ec\'edents, si l'on pose
$\widetilde \Phi_m(a) = \prd_{k=1}^g \ell_k^{n_k}$, $\ \,\ell_1 < \ell_2< \ldots < \ell_g$,\ \ $n_k \geq 1$, 
tous les premiers $\ell_k$ sont congrus \`a 1 modulo $m$ (car de degr\'e 1 et non rami\-fi\'es dans $\Q(\mu_m)/\Q$).
 Il en r\'esulte aussi que pour un tel $\ell =\ell_j$ 
(en posant $\ell -1 = t\,m$), $\ell $ est totalemment d\'ecompos\'e dans l'extension Galoisienne 
$\Q(\mu_{\ell-1})(\sqrt[t] a)/\Q$ puisque $a$ est localement 
de la forme $b^t$ modulo~$\ell$ ($\ell$~ne divise pas $a$ et n'est pas ramifi\'e dans cette extension).
Ces questions d'ordres modulo $\ell$ sont li\'ees \`a des techniques issues de la conjecture d'Artin
sur les racines primitives et de la d\'emonstration de Hooley, susceptibles de s'appliquer aux quotients 
de Fermat (voir \cite{Mo} pour un expos\'e exhaustif).

\begin{lemma}\label{dec1} On suppose $(m, p)$ distinct de $(2,2)$.
On a $p^2 \div \widetilde \Phi_m(a)$ si et seulement si $m=o_p(a)$ $\&$ $p^2\div  \Phi_{m}(a)$, donc
si et seulement si $m=o_p(a)$ $\&$ $q_p(a)=0$.
\end{lemma}

\begin{proof} 
En effet,  si $p^2\div  \Phi_{o_p (a)}(a)$, comme $p\div \Phi_{o_p (a)}(a)$ et $p \notdiv {o_p (a)}$,  on a
$\widetilde \Phi_{o_p (a)}(a) = \Phi_{o_p (a)}(a)$ et donc $p^2\div \widetilde \Phi_{m}(a)$.

R\'eciproquement, si $p^2\div \widetilde \Phi_m(a)$, on peut supposer que ${\rm p.g.c.d.}\,(\Phi_m(a), m)=r$
avec $m=r^e\, o_r (a)$, $e\geq 1$, sinon ${\rm p.g.c.d.}\,(\Phi_m(a), m)=1$, $\widetilde \Phi_m(a) = \Phi_m(a)$ et
n\'eces\-saire\-ment $m = {o_p (a)}$.
Ainsi $\widetilde \Phi_m(a) = \Frac{\Phi_m(a)}{r}$, donc
$p\notdiv m$ (i.e., $p\ne r$ car $r^2 \notdiv \Phi_m(a)$ par le Lemme \ref{lem3} qui exclue le cas $p^e=m=2$), 
d'o\`u $p^2 \div \Phi_m(a) = \Phi_{o_p (a)}(a)$.
\end{proof}

\begin{lemma}\label{dec2} 
Pour $a$ fix\'e, les $\widetilde\Phi_m(a)$, $m\geq 1$, sont premiers entre eux deux \`a deux. Pour tout $p\geq 2$ 
il existe un et un seul $m\geq 1$ (\'egal \`a $o_p (a)$), tel que $p \div \widetilde\Phi_m(a)$.
\end{lemma}

\begin{proof} Si $p\ne 2$ divise $\widetilde\Phi_m(a)$ et $\widetilde\Phi_{m'}(a)$, d'apr\`es le Th\'eor\`eme \ref{theorem} 
on a $m=p^e o_p (a)$ et $m'=p^{e'} o_p (a)$, $e, e' \geq 0$. Si par exemple $e\geq 1$, on a $p=r$ (absurde car $r^2$
ne divise pas $\widetilde\Phi_m(a)$)~; donc $e=e'=0$ et $m=m'$.

\smallskip
Si $p=2$, on obtient encore $m=2^e$, $m'=2^{e'}$, $e, e' \geq 0$~; 
le cas $e$ ou $e'$ $\geq 2$ \'etant impossible car alors $\widetilde\Phi_m(a)$ ou $\widetilde\Phi_{m'}(a)$ est
impair, il reste par exemple le cas $e=1$, $e'=0$, mais alors $\widetilde\Phi_2(a) = \frac{a+1}{2}$ et $\widetilde\Phi_1(a) = \frac{a-1}{2}$
qui ne peuvent \^etre tous deux divisibles par 2. Enfin tout $p$ divise $\Phi_{o_p (a)}(a)=\widetilde\Phi_{o_p (a)}(a)$. 
\end{proof}

En r\'esum\'e on a obtenu l'\'equivalence, plus forte que $q_p(a)=0 \Longleftrightarrow p^2 \div \Phi_{o_p (a)}(a)$~:

\begin{theorem} \label{wtheorem} Soit  $a \in \N \Sauf \{0, 1\}$ et soit $p$ premier. Alors $q_p(a)=0$ si et seulement 
si $p^2$ divise $\widetilde \Phi_{o_p (a)}(a)$.
\end{theorem}

Ainsi, la recherche des quotients de Fermat nuls est de nature
multiplicative, a priori diff\'erente de  celle des quotients de Fermat $1, 2, \ldots , p-1$~: si 
$\widetilde \Phi_m(a) = \prod_{k=1}^g \ell_k^{n_k}$,
le cas $q_{\ell_j}(a)=0$ se lit sur l'exposant $n_j$ tandis que si $n_j=0$ on a
 $\Frac{\widetilde  \Phi_{m}(a)}{\ell_j} = \prod_{k\ne j} \ell_k^{n_k}$ qui relie $q_{\ell_j}(a)$ au produit $\prod_{k\ne j} \ell_k^{n_k}$
 au moyen d'une congruence modulo $\ell_j$ convenable.
 
\subsection{Premi\`ere approche des questions de probabilit\'es}\label{manques}

 Pour chacun des cas $q_p(a)=u \in [0, p[$, la probabilit\'e est a priori voisine de $\frac{1}{p}$. 
 Des probabilit\'es inf\'erieures \`a $\frac{1}{p}$ en moyenne pour $u=0$ ne sont pas contradictoires avec une somme \'egale \`a 1 
 car une \'etude num\'erique montrera qu'environ $\frac{1}{3}$ des $u \in [0, p[$ ne sont pas de la forme $q_p(z)$, $z\in [2, p[$ 
 (pour $p=11$, on trouve que $u=3,6,8,9$  ne sont pas atteints). Par exemple, lorsque $a \ll p$ 
 ($a$ fix\'e) a un $p$-quotient de Fermat nul, alors tout $b \geq 2$
 tel que $a\,b<p$ v\'erifie $q_p(a\,b)=q_p(b)$, ce qui montre une ``non surjectivit\'e'' \'evidente.
Pour $p=1093$ et $p=3511$ ($q_p(2) = 0$), on obtient les proportions de $0.60348$ et $0.60285$,
 respectivement, de $u$ non atteints. 
  
\smallskip 
 On peut utiliser le programme suivant pour d'autres exp\'erimentations~:
 
 \smallskip
 \footnotesize
 $\{$$p=10^3; while(p<10^3+100, p=nextprime(p+1); P=0; p2=p^2; N=0.0; $\par$
 for(a=1, p-1, Q=Mod(a,p2)^{(p-1)}-1;$\par$
q=component(Q,2)/p; P=P+x^q); for(k=1,p, u=component(P,k);$\par$
if(u==0, N=N+1) ); print(p,"   ",N/(p-1)) )$$\}$\ \footnote{\,Dans tous les programmes PARI \cite{P} propos\'es, la compatibilit\'e avec TeX oblige \`a \'ecrire les symboles {\it par}, $\&$ avec un antislash, \`a placer  des {\it \$} et des {\it  \{  \}} pour les exposants\ldots
Sous r\'eserve d'\'eliminer ces symboles, le fichier tex permet de copier-coller ces programmes.}

\normalsize
 \smallskip
 En outre le cadre probabiliste pr\'ec\'edent de recherche des solutions $z \in [2, p[$ est plut\^ot de type ``densit\'e'' sur un l'intervalle
 tendant vers l'infini avec $p$~; or on verra au \S\,\ref{probas} que ces deux cas de figure sont \`a distinguer soigneusement.  
 
 \smallskip
 L'aspect chaotique de ces estimations invite \`a faire des statistiques cumul\'ees~: $a\geq 2$ et $u$ (en g\'en\'eral $0$) sont fix\'es 
 mais on teste plusieurs $p$, par exemple une dizaine, pour lisser le ph\'enom\`ene puisque, pour un seul
 $p$, plusieurs valeurs inconnues de $q_p(z)$, $z\in [2, p[$, sont de probabilit\'e nulle et d'autres multiples de $\frac{1}{p}$.

 \smallskip
Les cas o\`u $\widetilde \Phi_m(a)$ est divisible par le carr\'e d'un nombre premier $p$ sont rarissimes.
Rappelons cependant les toutes premi\`eres valeurs $(a, p)$ pour lesquelles $q_p(a)=0$, qui correspondent le plus souvent \`a des cas triviaux comme $p=2$ et $a\equiv 1 \pmod 4$, $p=3$ et $a\equiv 1, 8 \pmod 9$~:

\medskip
\footnotesize
$\{$$for(a=2, 14, p=0; while(p<100, p=nextprime(p+1); p2=p^2;$\par$
 Q= Mod(a,p2)^{(p-1)}-1; if(Q==0, print(a,"   ",p)) )) $$\}$

\medskip
 \normalsize
\centerline{ $(a, p) =  (3,   11); (5,   2); (7,   5); (8,   3); (9,   2); (9,   11); (10 ,  3); (11 ,  71); (13,   2); (14 ,  29)$. }

\begin{remark}\label{rema3} On utilise $\widetilde \Phi_m(a)$ au lieu de $\Phi_m(a)$ car en raison du nombre premier $r$
\'eventuel, les valeurs $\Phi_m(a)$ sont trivialement non premi\`eres entre elles
(pour les $m$ de la forme  $r^e .\,  o_r(a)$, $e = 0, 1, \ldots$)~; donc on ne peut pas \'etudier les facteurs carr\'es du produit formel
${\mathcal P}(a) := \prod_{m\geq 1} \Phi_m(a)$ qui contient pour chaque $r$ les sous-produits $\prod_{e\geq 1} \Phi_{r^e .\,  o_r(a)}(a)$ 
et donc les facteurs parasites $r^\infty$, ce qui n'est plus le cas de $\widetilde{\mathcal P}(a) :=\prod_{m\geq 1} \widetilde\Phi_m(a)$.
\end{remark}

\section{Premi\`ere analyse probabiliste pour $q_p(a) =0$} \label{section3}

\subsection{R\'esultat de A. Granville \cite{G}}\label{subrag} Ce r\'esultat a \'et\'e obtenu, dans le cas le plus g\'en\'e\-ral,
sous la conjecture $ABC$.
Soit $f\in \Z[x]$ un polyn\^ome tel que l'ensemble des $f(n)$, $n \in \Z$, ait un plus grand commun diviseur \'egal \`a 1
(le cas plus complet \'enonc\'e dans  \cite{G} ne s'applique pas pour nous). 

\begin{proposition} \label{gran}
La densit\'e naturelle des entiers $A \in \N$ tels que $f(A)$ est sans facteur carr\'e non trivial est donn\'ee par l'expression~:
$$\prd_{p \,{\rm premier}\, \geq 2}\  \Big( 1 - \Frac{c_p}{p^2} \Big) , \ \hbox{o\`u 
$c_p =\Big \vert\, \Big\{ b \in [0, p^2[,\ \, f(b) \equiv 0 \!\!\pmod {p^2} \Big\}\, \Big\vert$}, $$
chaque facteur $1 - \Frac{c_p}{p^2}$ \'etant la densit\'e (dite densit\'e locale associ\'ee \`a $p$)  des 
$A \in \N$ tels que $p^2 \!\notdiv f(A)$. Dans le cas local, la densit\'e des $A \in \N$ tels que $p^2\!\div f(A)$
\'etant~$\Frac{c_p}{p^2}$.
\end{proposition}

D'une certaine mani\`ere on peut dire que les \'ev\'enements $p^2 \notdiv f(A)$ sont ind\'ependants par rapport \`a $p$.

\subsection{Calcul des coefficients $c_p$ pour les polyn\^omes $\Phi_m(x)$, $m\geq 1$}\label{subcp}
Le p.g.c.d. des $\Phi_m(n)$, $n\in \Z$, est \'egal \`a 1 car $\Phi_m(0) = \pm 1$ puisque toute racine de l'unit\'e est de norme $\pm 1$.

\smallskip
Comme $\Phi_m(0) = \pm 1$, on a pour tout $p$ premier,
$$c_p =\Big \vert\, \Big\{ A \in [1, p^2[,\ \, \Phi_m(A) \equiv 0 \pmod {p^2} \Big\}\, \Big\vert. $$

\begin{proposition}\label{cp} Si $p \geq 2$ ne divise pas $m$, on a  $c_p = 0$ pour les $p\not\equiv 1 \pmod m$ et 
$c_p = \varphi(m)$ pour les $p \equiv 1 \pmod m$, o\`u $\varphi$ est l'indicateur d'Euler.

\smallskip
Si $m= p^e m'$, $e \geq 1$,  $p\notdiv m'$, on a $c_p =0$ sauf si $m=2$, auquel cas $c_2=1$.
\end{proposition}

\begin{proof} (i) Cas $p\notdiv m$.
Dans ce cas, la congruence $\Phi_m(A) \equiv 0 \pmod p$ est \'equivalente 
\`a $m=o_p(A)$ et on a $p\equiv 1 \pmod m$~; donc pour $p\notdiv m$, il y a exactement
$\varphi(m)$ nombres distincts $A_i \in [1, p[$ pour lesquels $\Phi_m(A_i) \equiv 0 \pmod p$. 

\smallskip
Consid\'erons pour $i$ fix\'e les entiers de la forme $A = A_i + \lambda_i \, p \in  [1, p^2[$ (i.e., $\lambda_i \in [0, p[$). On a
$\Phi_m(A) \equiv  \Phi_m(A_i)  + \lambda_i \, p \,\Phi'_m(A_i) \pmod {p^2}$, o\`u $\Phi'_m$ est le polyn\^ome d\'eriv\'e 
de $\Phi_m$~; d\`es que $\Phi'_m(A_i) \not \equiv 0 \pmod p$, il existe un unique $\lambda_i$ modulo $p$ 
donnant $\Phi_m(A) \equiv 0 \pmod {p^2}$ et dans ce cas, $c_p = \varphi(m)$.

\smallskip
Montrons que $\Phi'_m(A_i) \not \equiv 0 \pmod p$. On a $x^m - 1 = \Phi_m(x) \times Q(x)$, $Q \in \Z[x]$~; d'o\`u
$m\,x^{m-1} = \Phi'_m(x)  \times Q(x) + \Phi_m(x) \times Q'(x)$. Si $\Phi'_m(A_i) \equiv 0 \pmod p$ il vient
$m\,A_i^{m-1}  \equiv 0 \pmod p$~; comme $p\notdiv A_i$ par hypoth\`ese, on a $m \equiv 0 \pmod p$ (absurde).

\medskip
(ii) Cas o\`u $p=r\div m$. D'apr\`es le Lemme \ref{lem3}, $m=r^e  .\, o_r(A)$, $e\geq 1$, et $\Phi_m(A) \equiv 0 \pmod {r^2}$ 
n'a pas de solutions sauf si $m=2$, auquel cas $c_2=1$.
\end{proof} 

\subsection{Densit\'es et Probabilit\'es}\label{subdp}
De fa\c con g\'en\'erale, $A\in \N$ d\'esigne une variable et $F(A)$ une propri\'et\'e. On appelle alors densit\'e naturelle
(ou, pour simplifier, densit\'e) la limite (si elle existe),
$\lim\limits_{\substack{\ \ y \to \infty}}  \, \Frac{1}{y} \Big \vert \Big\{  A  \leq y,\ \,  F(A)  \Big\} \Big \vert$
(cf. \cite{T1}, III.1.1).

\smallskip
Si $F = F_p$ est la propri\'et\'e locale $p^2 \div f(A)$, la densit\'e est celle donn\'ee dans la Proposition \ref{gran}, \'egale \`a
$\Frac{c_p}{p^2}$ (celle de $p^2 \notdiv f(A)$ \'etant $1-\Frac{c_p}{p^2}$). Dans ce cadre, la densit\'e est relative \`a tous les entiers
(y compris ceux divisibles par $p$). 
Dans $\N\Sauf p\N$ ces densit\'es deviennent respectivement $\Frac{c_p}{p\,(p-1)}$ et $1- \Frac{c_p}{p\,(p-1)}$.

\smallskip
 Il faut distinguer la notion de densit\'e, relative \`a la propri\'et\'e~:

\medskip
\centerline {\it pour $p$ fix\'e,  $p^2 \div f(A)$ pour $A \in \N$ variant arbitrairement,}

\medskip
de celle de probabilit\'e d\'efinissant l'\'evenement~:

\medskip
\centerline {\it pour $a$ fix\'e, $p^2 \div f(a)$ pour $p$ premier variant arbitrairement}

\medskip
(cas de l'\'etude de $q_p(a)=0$ \'equivalent \`a
$p^2 \div \widetilde \Phi_{o_p(a)}(a)$, $p \notdiv a$ (Th\'eor\`eme \ref{wtheorem})).

\medskip
Analysons sur des cas pr\'ecis ce qu'il en est~; soit $d \div p-1$ un ordre fix\'e.

\smallskip
Si $p=2$ et $d=1$, $\Phi_1(x)= x-1$ et  la densit\'e des $A$ tels que $A-1 \equiv 0 \pmod 4$ est trivialement
$\frac {\varphi(1)}{p^2} = \frac{1}{4}$ (resp.  $\frac {\varphi(1)}{p\,(p-1)} =\frac{1}{2}$ pour les $A$ impairs). 
Ici l'ordre de grandeur de $a$ ne joue pas encore, mais si l'on veut par exemple $a<p$, la seule solution est $a=1$.

\smallskip
 Le cas $p=3$ est plus \'eloquent car pour $d=1$, la densit\'e des $A$ tels que $A-1 \equiv 0 \pmod 9$ est trivialement 
$\frac{1}{9}$ (resp. $\frac{1}{6}$ pour les $A$ \'etrangers \`a 3) et celle correspondant \`a $d=2$ (i.e., $\Phi_2(x)=x+1$)  
est aussi $\frac{1}{9}$(resp. $\frac{1}{6}$)~; puisque $A\not\equiv 0 \pmod 3$ peut \^etre d'ordre 1 ou 2 modulo 3,
 la densit\'e totale pour $q_3(A)=0$ est $\frac{2}{9}$ (resp. $\frac{1}{3}$). 

\smallskip
Par contre pour $a$ fix\'e non divisible par 3, le cas $a-1 \equiv 0 \pmod 9$ se produit une fois (solution minimale $a=1$) 
et le cas $a+1 \equiv 0 \pmod 9$ \'egalement, mais avec l'unique solution minimale $a=8$~; or si $a$ \'etait fix\'e
``assez petit'', la probabilit\'e correspondante chute ou si l'on pr\'ef\`ere, la probabilit\'e pour $q_3(a)$
d'\^etre non nul augmente. Le cas $a+1 \equiv 0 \pmod 9$ n'est donc plus envisageable avec une probabilit\'e 
\'egale \`a sa densit\'e $\frac{1}{6}$. Au total la probabilit\'e pour que $q_3(a)=0$ n'est plus la densit\'e totale 
$\frac{1}{6}+\frac{1}{6}=\frac{1}{3}$ (selon que $o_3(a) = 1$ ou $2$).

\smallskip
Pour $p=7$, on trouve, pour $A \in [1, 7^2[$, $A \not\equiv 0 \pmod 7$, les solutions
suivantes \`a $p^2 \div \widetilde \Phi_{d}(A)$, selon l'ordre $d$ modulo $p$ consid\'er\'e~: 

\smallskip
\centerline{$A = 1$\! ($d = 1$),  $\ A = 48$\! ($d = 2)$,  $\ A = 18, 30$\! ($d = 3$), $\ A =  19, 31$\! ($d = 6$). } 

\medskip
Pour $p=101$, on trouve de m\^eme~:

\smallskip
$A = 1$\! ($d = 1$),  $\ A = 181$\! ($d = 25)$,  $\ A = 248$\! ($d = 100$), \ldots, 

\hfill$\ A =  10020$\! ($d = 50$), $\ A =  10200$\! ($d = 2$). 

\medskip
On voit bien que si $a$ est fix\'e assez petit lorsque $p$ varie de fa\c con arbitraire, la probabilit\'e
de divisibilit\'e de $\widetilde \Phi_{d}(a)$ par $p^2$ peut m\^eme \^etre tr\`es faible. 

\smallskip
Pour simplifier, nous parlerons par abus de probabilit\'es lorsque $a$ est fix\'e, et nous \'ecrirons 
${\rm Prob}\big( f(a)\ {\rm s.f.c.} \big)$ et ${\rm Prob}\big(p^2 \notdiv  f(a)\big)$ respectivement, puis
${\rm Prob}\big( q_p(a)=0 \big)$, ${\rm Prob}\big( q_p(a)\ne0 \big)$, etc. 

\smallskip
A partir de ce principe et de ces observations num\'eriques, nous examinerons dif\-f\'erentes heuristiques en partant des plus faibles 
(permettant encore l'infinitude des $q_p(a)$ nuls) pour aller vers les plus fortes associ\'ees \`a la finitude des $q_p(a)$ nuls.

\smallskip
On peut donc d\'ej\`a admettre la premi\`ere heuristique g\'en\'erale suivante~:

\begin{heuristic} \label{heur1} {\it 
Supposons que pour $A\in \N$ (resp. $A \in \N\, \Sauf \, p \N$), la propri\'et\'e ``globale'' $F(A)$
(resp. la propri\'et\'e ``locale''  $F_p(A)$) soit du type $f(A)$ a un facteur carr\'e (resp. $p^2 \div f(A)$), $f\in \Z[X]$.
Alors la densit\'e correspondante dans $\N$ (resp. $\N\, \Sauf \, p \N$)
est un  {\it majorant} de ${\rm Prob} \big (F(a) \big)$ (resp. ${\rm Prob} \big (F_p(a) \big)$ pour $a$ fix\'e.}
\end{heuristic}

Par exemple, les densit\'es locales $\Frac{\varphi(d)}{p\,(p-1)}$, caract\'erisant la propri\'et\'e $F_p(A)$ d\'efinie par 
$p^2 \div \widetilde \Phi_{d}(A)$ pour les $A$ d'ordre $d \div p-1$, sont des {\it majorants} 
de ${\rm Prob} \big (q_p(a) = 0 \big)$ pour $a$ fix\'e de m\^eme ordre $d$ ($a, A \in \N\, \Sauf \, p \N$).
Ceci sera utilis\'e au \S\,\ref{probas}.

\smallskip
La Proposition \ref{cp} a la cons\'equence suivante concernant la densit\'e globale (on rappelle que 
$\widetilde\Phi_m(A)= \Phi_m(A)$ si ${\rm p.g.c.d.}\,(\Phi_m(A), m) = 1$, ou $\widetilde\Phi_m(A)= \frac{\Phi_{r^e .\, o_r(A)}(A)}{r}$ 
sinon, pour un unique nombre premier~$r$ et $e\geq 1$)~:

\begin{corollary}\label{coro1} Pour tout $m\ne 2$, la densit\'e des $A\in \N$ tels que $\widetilde\Phi_m(A)$ 
est sans facteur carr\'e non trivial est $\!\!\!\prd_{p \equiv 1\, ({\rm mod}\, m)} \Big( 1 - \Frac{\varphi(m)}{p^2} \Big)$.
Pour $m= 2$, la densit\'e des $\widetilde\Phi_2(A)= A+ 1$ ou $\frac{1}{2}(A+ 1)$ sans facteur carr\'e est
$\prd_{p \geq 2} \big( 1 - \Frac{1}{p^2} \big) = \Frac{6}{\pi^2} \approx 0.6$.
\end{corollary}

\begin{remark}\label{rema4}
 Les valeurs de $P_m=\prd_{p \equiv 1\, ({\rm mod}\, m)}  \big( 1 - \Frac{\varphi(m)}{p^2} \big)$ 
se calculent tr\`es facile\-ment par le programme suivant~:

\medskip
\footnotesize 
$\{$$ for(m=  1000002,   1000003,  f=eulerphi(m); P=1.0; $\par$
for(n=1, 2*10^6, p=1+n*m; if(isprime(p)==1,  P=P*(1 - f/p^2) )  ); print(m,"   ",P) ) $$\}$

\normalsize
qui conduit au tableau~:
\footnotesize 
\begin{align*}
P_3& \approx  0.93484202308683713466409790668210927326\\
 P_4& \approx  0.89484123120292308233007546174564683811\\
 P_5 &\approx  0.95709281951397098677511212591026189432\\
P_{39}& \approx  0.99466134034387664509206853899643846793\\
P_{40}& \approx  0.98961654058761399079945594714123081337 \\
P_{10003} &\approx  0.99999392595496021757107201755865536021\\
P_{1000002}& \approx  0.99999964016779551958062234579864526853\\
\end{align*}
\end{remark}

\normalsize
\subsection{Densit\'es et probabilit\'es au niveau des $p$-quotients de Fermat}\label{probas}

Soit  $a \in \N \Sauf \{0, 1\}$ fix\'e. On \'ecrit que la probabilit\'e d'avoir $q_p(a) = 0$ est de la forme
${\rm Prob} \big (q_p(a) = 0 \big) = \Frac{1}{p^{1+ \epsilon(p, a)}}$, avec $\epsilon(p, a)$ voisin de 0. 
 
 \smallskip
Dans l'\'etude probabiliste de la condition $q_p(a)= 0$, $p$ est variable tendant vers l'infini de sorte que l'on a $a<p$
pour tout $p$ assez grand~; on va donc rechercher, comme expliqu\'e au \S\,\ref{subdp} (cf. Heuristique \ref{heur1}), 
la densit\'e locale associ\'ee qui consti\-tuera un majorant de la probabilit\'e correspondante.

\smallskip
Soit $u \in [0, p[$ donn\'e.
La densit\'e des $A$ \'etrangers \`a $p$ tels que $q_p(A)= u$ se lit aussi dans l'intervalle $[0, p^2[$ puisque $q_p(A+\Lambda\,p^2)\equiv q_p(A) \pmod p$ pour tout entier $\Lambda$.

\begin{lemma} \label{lambda}
Soit $z \in [1, p[$, $p$ premier~; alors il existe un unique $\lambda_u(z) \in [0, p[$ tel que $Z = z + \lambda_u(z)\,p \in [1, p^2[$
v\'erifie $q_p(Z)= u$. Le nombre $\lambda_u(z)$ est caract\'eris\'e par la congruence $\lambda_u(z) \equiv  z\,(q_p(z) -  u )\pmod p$
et on obtient $Z\equiv z^p -z u\,p \pmod{p^2}$.

Par cons\'equent, la densit\'e des $A\in \N \Sauf p\N$ tels que $q_p(A)= u$ est \'egale \`a $\Frac{1}{p}$.
\end{lemma}

\begin{proof} Pour tout $\lambda \in\N$, $(z + \lambda \,p)^p - (z + \lambda \,p) \equiv z^p - z - \lambda \,p \!\pmod {p^2}$,
d'o\`u $\lambda \equiv z\,q_p(z) - Z\,q_p(Z)\equiv z\,q_p(z) - z\,q_p(Z) \pmod p$. Donc $q_p(Z) =u$ si et seulement si $\lambda = \lambda_u(z)\equiv z\,q_p(z)-z\,u \pmod p$. On a donc pour chaque $z \in [1, p[$ un unique 
$Z = z + \lambda_u(z)\,p \in [1, p^2[$ tel que $q_p(Z) =u$, d'o\`u la densit\'e ($Z$ est aussi le r\'esidu modulo $p^2$
de $z^p -z u\,p$). Pour $u=0$, $Z$ est le r\'esidu modulo $p^2$ de  $z^p$.
\end{proof}

Rappelons que $q_p(A) = 0$ est \'equivalent \`a $p^2 \div \widetilde \Phi_{o_p(A)}(A)$ (Th\'eor\`eme \ref{wtheorem}).
D'apr\`es les r\'esultats ``locaux'' (cf. \S\S\,\ref{subrag}, \ref{subcp}, Corollaire \ref{coro1}), 
la densit\'e des $A \in \N\Sauf p\N$ tels que $p^2 \div \widetilde \Phi_m(A)$ est
\'egale \`a $\Frac{\varphi(m)}{p\,(p-1)}$ (resp. $1$) si $m = o_p(A)$ (resp.  $m \ne o_p(A)$).  
En faisant la somme sur les ordres possibles, on retrouve bien la densit\'e $\sm_{d \div p-1}\Frac{\varphi(d)}{p\,(p-1)} = \Frac{1}{p}$. 
Revenons au cas d'un entier $a \geq 2$ fix\'e pour lequel la probabilit\'e d'avoir $q_p(a)=0$ est a priori major\'ee par $\Frac{1}{p}$.
On a facilement $o_p(a)> \Frac{{\rm log}(p)}{{\rm log}(a)}$
 puisque $a^{o_p(a)} = 1 + \lambda \, p$, $\lambda \geq 1$, et de fait ${\rm Prob}\big(o_p(a) = d \big)=0$ 
pour les $d<\Frac{{\rm log}(p)}{{\rm log}(a)}$.

\smallskip
Pour $a \in \N \Sauf p\N$ fix\'e, on a $o_p(a) \in \{d,\ d \div p-1\}$ et une heuristique raisonnable est que 
la probabilit\'e correspondante est major\'ee par la densit\'e relative \`a la propri\'et\'e locale $o_p(A) = d$, qui est
\'egale \`a $\Frac{\varphi(d)}{p-1}$, car $A$ 
n'est pas divisible par $p$ et seul le r\'esidu de $A$ dans $[1,p[$ intervient sachant qu'il y a exactement 
$\varphi(d)$ \'el\'ements d'ordre $d$ dans cet intervalle. 
Mais le ph\'enom\`ene pr\'ec\'edent sur les petites valeurs de $d$ rend les ``grands'' ordres plus probables pour $a$, 
ce qui semble pouvoir \^etre n\'eglig\'e dans la mesure o\`u, pour $h = \frac{{\rm log}(p)}{{\rm log}(a)}$, on a
$\sum_{d<h} \frac{\varphi(d)}{p-1} < O(1) \,\frac{{\rm log}^2(p)}{p}$.

\begin{remark}
Soient $a$ fix\'e et $p$ arbitrairement grand~; on a alors le ph\'eno\-m\`ene analogue suivant~: soit $g>a$
 et soit $G := \Big\{g^i, \  1 \leq  i< \Frac{{\rm log}(p)}{{\rm log}(g)} \Big\}  \subseteq [2, p[$.
Cet en\-semble est constitu\'e d'\'el\'ements plus grands que $a$, dont les ordres sont certains diviseurs 
$\delta$ de $p-1$, et ceci modifie le d\'ecompte pour $a$, ce qui fait que, a priori, ${\rm Prob} \big (o_p(a) = \delta \big)$
est inf\'erieur \`a $\Frac{\varphi(\delta)}{p-1}$.
\end{remark}

\begin{example} \label{ex1}
Prenons $p= 37813$, $a=2$~; alors pour  $g=3$, on a $G = \{ 3, 9, 27, 81$, $243, 729, 2187,$ $6561, 19683  \}$
dont les \'el\'ements sont d'ordres respectifs $18906$, $9453$, $6302$, $9453$, $18906$, $3151$, 
$18906$, $9453$, $6302$.
Pour $g=5$ on trouve les ordres $37812, 18906, 12604, 9453, 37812, 6302$.
On peut construire de tels ensembles jusqu'\`a $g= 193$ (donnant les ordres $37812, 18906$).

\smallskip
Donc pour $a=2$ (d'ordre $p-1=37812$), la probabilit\'e ne peut co\"incider avec la densit\'e $\frac{\varphi(p-1)}{p-1}
= 0.3165$. Le ph\'enom\`ene est difficile \`a quantifier, mais a une influence importante. 
\end{example}

La probabilit\'e correspondante de nullit\'e de $q_p(a)$, pour $a$ fix\'e et $p$ variable, est donc a priori fortement major\'ee par
$\sm_{d \div p-1} \Frac{\varphi(d)}{p-1} \times  \Frac{\varphi(d)}{p\,(p-1)}
 =  \Frac{1}{p\,(p-1)^2} \sm_{d \div p-1} \varphi(d)^2$.

\smallskip
En r\'esum\'e on a obtenu dans ce premier cadre le r\'esultat heuristique suivant~:

\begin{heuristic} \label{heur2}  {\it 
On a, pour $a\in \N\Sauf \{0, 1\}$ fix\'e et $p$ assez grand~:

\smallskip
\centerline{${\rm Prob} \big (q_p(a) = 0 \big) := \Frac{1}{p^{1+ \epsilon(p,a)}}  < \Frac{1}{p\,(p-1)^2} \sm_{d \div p-1}\varphi(d)^2$, }
 
ou de fa\c con \'equivalente}
$\ \epsilon(p, a)  >  \Frac{1} {{\rm log}(p)} \, \Big( 2\, {\rm log} (p-1) - {\rm log} \Big( \sm_{d \div p-1}\varphi(d)^2 \Big) \,\Big)$.
\end{heuristic}

\begin{remark} \label{remaeps}
Si l'heuristique pr\'ec\'edente est v\'erifi\'ee, alors on obtient~:

\smallskip
\centerline {$\epsilon(p, a)  > 0$ car $\Frac{1}{p^{1+\epsilon(p,a) }} < \Frac{\sum_{d \div p-1}\varphi(d)^2}{p\,(p-1)^2} < \Frac{\big (\sum_{d \div p-1}\varphi(d)\,\big)^2}{p\,(p-1)^2} = \Frac{1}{p}$}

\smallskip
 (o\`u $\Frac{1}{p}$ est la densit\'e des $A$ tels que $q_p(A)=0$). 
Autrement dit, si $\upsilon(p,a) = \upsilon(p)$ est la fonction
$\upsilon(p) = \Frac{1} {{\rm log}(p)} \, \Big( 2\, {\rm log} (p-1) - {\rm log} \big( \sm_{d \div p-1}\varphi(d)^2 \big) \,\Big)$,
on a $\epsilon(p,a) > \upsilon(p) >0$ pour tout $p$ assez grand.
Afin de proposer de telles fonctions $\epsilon(p,a)$, nous allons donner une condition suffisante de convergence des s\'eries du type
$\sm_p\Frac{1}{p^{1+ \epsilon(p,a)}}$, la s\'erie $\sm_p\Frac{1}{p^{1+ \upsilon(p)}}$ ne l'\'etant pas comme l'a montr\'e 
G. Tenenbaum (cf.~\S\,\ref{subdiv}).
\end{remark}

\subsection{Une s\'erie de r\'ef\'erence convergente sur les nombres premiers} \label{subconv}
Pour tout $n \geq 1$, d\'esignons par $p_n$ le $n$-i\`eme nombre premier. 

\begin{lemma} \label{infty}
Soit $C>1$ une constante et soit $\eta(p) := C\, .\, \Frac{ {\rm log}_3 (p)}{ {\rm log} (p)}$, o\`u 
${\rm log}_k$ d\'esigne le $k$-i\`eme it\'er\'e de la fonction ${\rm log}$. Alors on a
$S := \sm_{p \geq 2} \ \Frac{1}{p^{1 + \eta(p)}} < \infty$.
\end{lemma}

\begin{proof} 
On a $\sm_{p \geq 2}\ \Frac{1}{p^{1 + C . {\rm log}_3 (p) / {\rm log} (p) }} =  \sm_{p \geq 2} \ \Frac{1}{p\,.\, {\rm log}^C_2 (p) }
= \sm_{n \geq 1} \  \Frac{1}{p_n\,.\, {\rm log}^C_2 (p_n) }$.

\smallskip
On sait que $p_n > n\, {\rm log}(n)$ (th\'eor\`eme de Rosser)~; donc on peut \`a une constante additive pr\`es
majorer $S$ par
$\sm_{n \geq n_0}\  \Frac{1}{n\, {\rm log}(n)\,.\, {\rm log}^C_2 (n\, {\rm log}(n)) }  <
\sm_{n \geq n_0} \ \Frac{1}{n\, {\rm log}(n)\,.\, {\rm log}^C_2 (n) }$
qui a m\^eme comportement que $\displaystyle\int_{x_0}^\infty  \Frac{dx}{x\, {\rm log}(x)\,.\, {\rm log}^C_2 (x) } =
\displaystyle \int_{y_0}^\infty  \Frac{dy}{y\,.\, {\rm log}^C(y) } < \infty$.
\end{proof}

 Cependant il ne faut pas oublier que $\epsilon(p,a) > \upsilon(p)$ 
et que par cons\'equent $\epsilon(p,a) > \eta(p)$ reste largement possible.
La diff\'erence entre $\upsilon(p)$ (situation divergente) et $\eta(p)$ (situation convergente) est tr\`es faible comme le montrent
les r\'esultats num\'eriques suivants pour $p$ tr\`es grand (avec $C=1.1$)~:

\smallskip
\footnotesize 
$\{$$for(n=10^{40}, 10^{40}+400, p=1+2*n; if(isprime(p) ==1, S=0.0; D=divisors(p-1);  $\par$
ND=numdiv(p-1); for(k=1, ND, d=component(D,k); f=eulerphi(d); S=S+f^2); $\par$
E= 1.1*log(log(log(p))) / log(p); U= (2*log(p-1) - log(S))/log(p); print(E-U,"    ",p ) ))$$\}$
\begin{align*}
eta-upsilon =  0.009409   \ \ \  p=20000000000000000000000000000000000000219\\
eta-upsilon =  0.004175   \ \ \   p=20000000000000000000000000000000000000231\\
eta-upsilon =  0.011358   \ \ \  p= 20000000000000000000000000000000000000243 \\
eta-upsilon =  0.008018   \ \ \  p= 20000000000000000000000000000000000000477\\
eta-upsilon =  0.005724   \ \ \  p= 20000000000000000000000000000000000000513\\
eta-upsilon =  -0.00386   \ \      p= 20000000000000000000000000000000000000593\\
eta-upsilon =  0.009301    \ \ \ p= 20000000000000000000000000000000000000723
\end{align*}
\normalsize
Le cas de ``croisement des courbes'' correspond par exemple au cas o\`u $p-1$ est divisible par un tr\`es grand nombre premier
donnant un grand $\varphi(d)$. Ci-dessus, on a le cas de
$p-1 = 2 ^4  \times  3^2 \times 11 \times 13   \times 971250971250971250971250971250971251$.

\subsection{Premi\`ere estimation majorante du nombre de solutions $p$ \`a $q_{p}(a)= 0$}\label{subdiv} 
Une estimation majorante du nombre de $p \leq x$
tels que $q_{p}(a)= 0$ est $\sm_{p \leq x} \Frac{1}{p^{1+\upsilon(p)}}$. 
Or la s\'erie $S := \sm_p \Frac{1}{p^{1+\upsilon(p)}}= 
\sm_p \Frac{1}{p\,(p-1)^2}\sm_{d \div p-1}\varphi(d)^2$, comme on pouvait s'y at\-tendre, est divergente, 
et G. Tenenbaum a d\'emontr\'e que 
$$S(x) := \sm_{p \leq x} \Frac{1}{p\,(p-1)^2}\sm_{d \div p-1}\varphi(d)^2 = O({\rm log}_2 (x))$$ 
lorsque $x \to \infty$ (cf. \cite{T2}). Sa d\'emonstration repose, entre autres, sur le th\'eor\`eme 
de Bombieri--Vinogradov rappel\'e dans \cite{T1} (Th\'eor\`eme II.8.34).
On en d\'eduit que pour $a$ arbitraire fix\'e le nombre moyen de solutions $p \notdiv a$ \`a  $q_p(a)=0$ v\'erifie~:
$$\Big\vert \Big\{ p \leq x,\  q_p(a)=0  \Big\vert \Big\} < O( {\rm log}_2 (x)) < \Frac{1}{2}\, {\rm log}_2 (x) $$ 

pour $x \to \infty$, apr\`es une estimation de la constante, ce qui reste une croissance tr\`es faible mais 
ne permet pas de conclure dans le cas de $a$ fix\'e une fois pour toutes (pour $x=10^8$, $S(x) \approx 1.3380$
et $\frac{1}{2}\, {\rm log}_2 (x) \approx 1.4567$).

\smallskip
La divergence de $\sm_p \Frac{1}{p^{1+\upsilon(p)}}$ n'est pas contradictoire avec une convergence \'eventuelle 
de $\sm_p \Frac{1}{p^{1+\epsilon(p,a)}}$ puisque chaque terme de $S$ est un majorant strict de 
${\rm Prob} \big (q_p(a) = 0 \big)$ (i.e., $\epsilon(p,a) > \upsilon(p)$ pour tout $p$ assez grand), voire un majorant d'un ordre 
de grandeur important, et il conviendra de revenir sur ce point, ce qui sera fait Section~\ref{section4} en partant du point de vue
heuristique de l'existence d'une loi de probabilit\'e binomiale sur le nombre de solutions \`a $q_p(z)=0$ pour $z\in [2, p[$.

\begin{remark}\label{rema11}
Comme expliqu\'e au \S\,\ref{subdp}, le fait que $A \in \N$ ne soit pas born\'e dans les calculs de densit\'es est fondamental 
puisque d\'ej\`a les $A$ qui sont de la forme $A= 1 +k \, (p_1 p_2 \cdots p_n)^2$ (o\`u les $p_i$ sont 
des nombres premiers distincts) conduisent \`a $q_{p_i}(A)= 0$ pour tout~$i$, et il y a bien d'autres fa\c cons de 
cr\'eer des $A$ avec beaucoup de $q_p(A)=0$, tout ceci ``comptant'' dans une estimation du nombre de solutions $p$. 

\smallskip
En effet, pour chaque $p \in \{p_1, \ldots, p_n\}$
soit $(B_p^j)_{j=1, \ldots, p-1}$ la famille des $p-1$ solutions canoniques $B_p^j\in [1, p^2[$ \`a $q_p(B_p^j) =0$
(cf. Lemme \ref{lambda})~;
alors tout $A$ satisfaisant \`a l'un des syst\`emes de congruences~:
\begin{align*}
A &\equiv B_{p_1}^{j_1} \!\!\pmod {p_1^2} ,\ \ j_1 \in \{1, \ldots,p_1-1\}     \\
&  \ldots  \\
A &\equiv B_{p_n}^{j_n} \!\! \pmod {p_n^2}  ,\ \ j_n \in \{1, \ldots,p_n-1\}  
\end{align*}
conduit \`a $q_{p_1}(A)= \cdots= q_{p_n}(A)= 0$, et c'est en outre une \'equivalence. 
Naturel\-lement $A$ devient en g\'en\'eral tr\`es grand.
\end{remark}

\begin{example}\label{ex2}
Pour $p_1=5$, $p_2=7$, on obtient les 24 solutions fondamentales modulo $35^2$~:

\medskip
$\{1, 18, 68, 99, 226, 276, 293, 324, 374, 393, 557, 607, 618, 668, 832, 851$, 

\hfill $901, 932, 949,  999, 1126, 1157, 1207, 1224\}$,

\smallskip
la plus petite solution $a >1$ de ce type \'etant $18$. 
\end{example}

\subsection{Quotients de Fermat non nuls sur un intervalle -- Exemples}\label{ex3}
 Un des aspects du probl\`eme de la finitude ou non des quotients de Fermat nuls est qu'il n'est pas rare de trouver  des valeurs de $a$ pour lesquelles $q_p(a) \ne 0$ sur un intervalle $p \in [2, B[$ o\`u $B$ est de l'ordre de $10^{10}$, ce qui accr\'edite la finitude. 

\smallskip
Or s'il existe effectivement des $a$ tels que $q_p(a) \ne 0$ pour tout $p$, un tel cas de finitude (triviale)
pour $q_p(a) = 0$ pourrait vouloir dire que tous les entiers $a \in \N\Sauf \{0, 1\}$ ont un nombre fini de quotients de Fermat nuls, une heuristique naturelle \'etant que l'on ne peut avoir deux cat\'egories de nombres fondamentalement diff\'erentes. 

\smallskip
On abordera cette existence  (sous les heuristique pr\'ec\'edentes et les r\'esultats de densit\'e) au Th\'eor\`eme \ref{theo2}
par un calcul effectif de densit\'e.

\medskip
Pour $2 \leq a \leq 100$ on trouve les exemples suivants (le cas $p=2$ \'eliminant tous les $a\equiv 1 \pmod 4$, $p=3$ \'eliminant tous les $a\equiv 1,  8 \pmod 9$, etc.)~:

\smallskip
Pour $a=34$ la premi\`ere solution est  $p=46145917691$.

\smallskip
Pour $a=66$, on trouve la premi\`ere solution $p = 89351671$.

\smallskip
Pour $a=88$, on trouve la premi\`ere solution $p = 2535619637$.

\smallskip
Pour $a=90$, on trouve la premi\`ere solution $p =  6590291053$.

\smallskip
Pour $a=47$ et $a=72$ on ne trouve aucune solution pour $p \leq 10^{11}$.

\smallskip
Dans \cite{KJ} on trouve les exemples suivants pour $a \in [2, 101]$ et $p \leq 10^{11}$~:

\smallskip
$(a,p) = (5, 6692367337), (23, 15546404183), (37, 76407520781), (97, 76704103313)$ et la solution remarquable
$(5, 188748146801)$, ce qui semble indiquer que la finitude \'eventuelle des $q_p(a)=0$ n'implique pas n\'ecessairement l'existence d'une borne, pour $p$, fonction de $a$.  

\smallskip
\smallskip
On peut poursuivre cette \'etude au moyen du programme suivant (par tranches)~:

\medskip
\footnotesize 
$\{$$A=47; p=10^{11} + 1; while(p < 2*10^{11}, p=nextprime(p+2); $\par$
 Q=Mod(A,p^2)^{(p-1)};  if(Q==1, print(p) ))$$\}$

\medskip
 \normalsize
De $a = 100000$ \`a $100099$, les r\'esultats sont similaires mais avec une rar\'efaction certaine,
car $a$ est fix\'e mais plus grand que dans le cadre classique ($a= 2, 3, \ldots$).

\smallskip
Jusqu'\`a $p< 10^8$, aucune solutions pour $a =100014, 100015,100022,
100030, 100055$, $100062, 100075, 100083$.

\smallskip
Pour d'autres exemples num\'eriques voir \cite{KJ}.

\section{Seconde analyse probabiliste pour $q_p(a) =0$} \label{section4}

L'approche pr\'ec\'edente (Section \ref{section3}), de type ``estimations de densit\'es'' relativement \`a la variable enti\`ere $A$,
ne tient pas assez compte du fait que l'on \'etudie $q_p(a)$ pour $a$ fix\'e ``petit'' et $p$ variable arbitrairement grand. 
Or, comme on l'a vu, le simple fait que $q_p(a)=0$
pour $p\gg a$ entra\^ine de nombreuses solutions dans $[2, p[$, puisque $q_p(a^j)=0$ pour $1 \leq j < 
\frac{{\rm log}(p)}{{\rm log}(a)}$ (avec $a^j \in [2, p[$).
D'o\`u la n\'ecessit\'e d'une premi\`ere \'etude sur l'intervalle $[2, p[$, \'etude qui ne d\'epend alors que de $p$.

\subsection{Etude des solutions \`a $q_p(z)=0$, $z \in [2,  p [$} \label{subdens} Dans cette partie nous allons essayer de justifier
l'existence d'une loi de probabilit\'e classique en utilisant un certain nombre d'arguments th\'eoriques et des
calculs num\'eriques.

\subsubsection{Retour sur l'aspect densit\'es vs probabilit\'es}\label{dens}
Soit $p$ un nombre premier fix\'e. Pour chaque $z \in [1, p[$ il existe un unique $\lambda(z) \in [0, p[$ tel que
$Z := z + \lambda(z) \, p \equiv z^p \pmod{p^2}$ v\'erifie $q_p(Z)=0$, d'o\`u la densit\'e des $A \in\N \Sauf p\N$
tels que $q_p(A)=0$  (pour $p$ fix\'e), \'egale \`a $\frac{1}{p}$. Ceci a \'et\'e vu \S\,\ref{probas} o\`u 
le Lemme \ref{lambda} d\'emontre une certaine \'equir\'epartition puisque la densit\'e des $A \in \N \Sauf p\N$ tels que $q_p(A)= u$ 
est aussi \'egale \`a $\frac{1}{p}$ quel que soit $u \in [0, p[$.
Autrement dit, si l'on fixe provisoirement $p$, pour $A \in [1, p^2[$ la probabilit\'e d'avoir $q_p(A)= u$ devient exactement \'egale \`a
la densit\'e $\frac{1}{p}$.

\begin{remark} Si $a$ est fix\'e et si $h$ est la partie enti\`ere de $\frac{{\rm log}(p)}{{\rm log}(a)}$, 
on a pour $j=1,\ldots,h$, $a^j \in [2, p[$ et $q_p(a^j) \equiv j\,q_p(a) \pmod p$.
Si $q_p(a)=0$, tous les $q_p(a^j)$ sont nuls, mais si $q_p(a)=u \ne 0$, on a $q_p(a^j) \equiv j\,u \pmod p$~; 
ces quotients de Fermat sont alors tous distincts et non nuls modulo $p$.
\end{remark}

On verra au moyen des exemples num\'eriques ci-apr\`es (cf. \S\,\ref{num1}) que le nombre de cas 
o\`u $q_p(z)=0$ pour $z\in [2, p[$
est statistiquement tr\`es faible (quelques unit\'es quelle que soit la taille de $p$)~; naturellement
il existe des cas exceptionnels~: lorsqu'une solution $z$ v\'erifie $z \ll p$, on a un certain nombre
de puissances de $z$, solutions dans $[2, p[$, mais on peut supposer que ceci est compens\'e par le fait que 
$Z \ll  p$, pour l'\'el\'ement correspondant $Z=z+\lambda(z)\,p \in [2, p^2[$, est d'autant moins probable. 
Si l'on se base sur l'existence d'une loi de probabilit\'e telle que ${\rm Prob}\big( \lambda(z) = 0 \big) <\frac{1}{p}$ (\`a comparer \`a 
${\rm Prob}\big(q_p(A) = 0 \big) = \frac{1}{p}$ pour $A \in [2, p^2[$),
on est fond\'e \`a \'ennoncer l'heuristique suivante qui semble l\'egitime au vu du faible nombre moyen de solutions pour chaque~$p$~:

\begin{heuristic} \label{heur33} {\it 
Les $p-2$ valeurs $Z = z + \lambda(z)\,p \equiv z^p \pmod{p^2}$, $z \in [2, p[$, $\lambda(z) \in [0, p[$, 
telles que $q_p(Z)=0$, sont al\'eatoires et ind\'ependantes dans $[2, p^2[$. 
Ceci est \'equivalent \`a la propri\'et\'e analogue pour les $p-2$ valeurs $\lambda(z) \in[0, p[$.}
\end{heuristic}

Une \'etude num\'erique montre clairement que le nombre de cas o\`u $\lambda (z)=0$ (i.e., $q_p(z)=0$)
est tr\`es faible car il correspond \`a une probabilit\'e voisine de $\frac{1}{p}$ au plus pour chaque $z$ (loi binomiale
de param\`etres $(p-2, 1/p)$,
cf. Heuristique \ref{heur4} et Remarque \ref{rema=1}).
Comme il y a $p-2$ solutions $Z \in [2, p^2[$, on peut s'attendre en moyenne \`a une solution $z\in [2, p[$
et \`a $p-3$ solutions $Z \in [p+1, p^2[$.

\smallskip
De m\^eme que pour les valeurs de $q_p(z)$, non toutes r\'ealis\'ees dans $[0, p[$ (cf. \S\,\ref{manques}), 
les nombres $\lambda (z) \in [0, p[$ tels que $q_p(z+\lambda (z)\,p) = 0$ ne sont pas tous atteints 
(il y a aussi environ $\frac{1}{3}$ des valeurs dans ce cas), 
ce qui est compatible avec le fait que en moyenne ${\rm Prob}\big(\lambda (z) = v \big) < \frac{1}{p}$ pour $v\in [0, p[$ 
(pour $p=11$, les $v =1,4,5,6,9$ ne sont pas atteints).

\subsubsection{Recherche num\'erique des solutions $z \in[2, p[$}\label{num1}
 Consid\'erons le programme suivant pour une tranche $B< p < B+200$~; pour chaque solution $z \in[2, p[$, 
 on indique l'ordre $d$ de $z$~:

\smallskip
\footnotesize 
$\{$$B=10^7; p=B;  while(p<B+200, p=nextprime(p+2); print(p); p2=p^2;  for(z=2, p-1, $\par$
Q=Mod(z,p2)^{(p-1)}-1;  if(Q==0, d=znorder(Mod(z,p)); print("      ",z,"   ",d) )) )$$\}$

\medskip
 \normalsize
Pour de grandes valeurs de $p$, on obtient peu de solutions comme attendu~:
{\footnotesize 
\begin{align*}
p= 10000019& \\
p= 10000079& \\
    &  z_1=6828481, \ \ \  d=  909098,     \\
     & z_1=9659873,  \ \ \  d=   5000039 ,    \\
p= 10000103&  \\
        &z_1= 4578211,   \ \ \  d=  386     ,     \\
     & z_1=4215058,   \ \ \  d=  10000102 ,    \\   
      &z_2=4732368,  \ \ \  d =   10000102 ,    \\
     &z_3= 8804922,   \ \ \  d=  10000102  ,  \\
p= 10000121&  \\
    & z_1= 1778643,   \ \ \  d=  10000120 ,     \\
    & z_1= 3601025 , \ \ \  d=   5000060 ,   \\
p= 10000139&  
\end{align*} }
\normalsize
Pour $p =1110000127$ (pris au hasard), il y a l'unique solution $z =723668846$~; le nombre premier suivant, $p=1110000149$,
donne 0 solutions dans $[2, p[$.

\smallskip
Ceci est assez analogue au cas des petits nombres premiers (nous omettons les $p = 2,3,5,7$,
$13, 17,19,23,31,41$ ne conduisant \`a aucune solution dans $[2, p[$)~:

{\footnotesize
\medskip
$p= 11$ ($z_1=3, \ d= 5,\  z_2 =9 , \ d=5$)~;  
$p= 29$ ($z_1 = 14 , \ d= 28$)~; 
$p= 37$ ($z_1 = 18 ,\  d=36$)~; 
$p= 43$ ($z_1 =  19 , \ d= 42$). }

\normalsize
\medskip
En outre les solutions $z \in [2, p[$ telles que $q_p(z)=0$ sont assez bien r\'eparties comme le v\'erifie le programme suivant qui 
compte (sur l'ensemble des $p<B$) le nombre $N_t$ de solutions sur un intervalle de longueur $(p-1)/t$, o\`u $t$ est une constante 
ajustable (ind\'ependante de $p$)~; on compare $N_t$ \`a $\frac{N}{t}$, o\`u $N$ est le nombre de solutions sur $[2, p[$. 
Les nombres $N_t$ et $N$ sont cumul\'es sur l'ensemble des $p$ car comme on vient de le voir, le nombre de solutions 
pour chaque $p$ est trop faible~:

\medskip
\footnotesize 
$\{$$B=10^6; N=0; t=25.0; Nt=0; p=1; while(p<B, p=nextprime(p+2); $\par$
p2=p^2; for(z=2, p-1, Q= Mod(z,p2)^{(p-1)}-1; if(Q==0, N=N+1; $\par$
 if(z<(p-1)/t, Nt=Nt+1)))); print(Nt,"   ",floor(N/t))$$\}$

\medskip
 \normalsize
On constate une bonne \'equir\'epartition en d\'epit de la m\'ethode utilis\'ee~; par exemple, pour $B=2\,.\, 10^5$, 
on trouve $N_t = 730$ pour une moyenne $\frac{N}{t}$ \'egale \`a $718$. 

\smallskip
D'autres exp\'erimentations num\'eriques 
montrent le ph\'enom\`ene suivant. On calcule (sachant que $\lambda(z) +\lambda(p-z) = p-1$) 
les quantit\'es $\sigma_n(p) := \Frac{2\,(n+1)}{(p-1)^{n+1}}  \sum_{z=1}^{(p-1)/2} \lambda(z)^n$
pour tout $n \geq 1$, o\`u l'on rappelle que $q_p(z+\lambda(z)\,p)=0$. 
On obtient alors une remarquable convergence altern\'ee vers 1~:

\medskip
\footnotesize 
 $\{$$n=11; for(h=1, 5, p=nextprime(10^7+1000*h); p2=p^2;  lambda = 0.0; $\par$
for(z=1, (p-1)/2, Z=Mod(z,p2); B=Z^p-Z; C=component(B,2)/p; $\par$
lambda=lambda+C^n);  print(p,"   ",2*(n+1)*lambda/(p-1)^{(n+1)}) )$$\}$

\medskip
$p=10001009 \ \ \ \sigma_{11}(p) = 1.0000467276683123307757138472299832521 $\par$
p=10002007  \ \ \  \sigma_{11}(p) = 1.0013551929880908863082167239611802354 $\par$
p=10003001  \ \ \  \sigma_{11}(p) = 1.0003688721711444598035617327427726537 $\par$
p=10004017  \ \ \  \sigma_{11}(p) = 0.9996190495531549422360323290673549366 $\par$
p=10005007  \ \ \  \sigma_{11}(p) = 0.9987657593324465195103425458241420008 $
\normalsize

\subsubsection{Classement des nombres premiers $p$ par nombre de solutions $z \in [2, p[$}\label{num2}
Le programme suivant (d'ex\'ecution assez longue) calcule les proportions de nombres premiers $p$ pour lesquels 
on a exactement 0, 1, ou 2 solutions, puis lorsque l'on a au moins 3 solutions $z\in [2, p[$ telles que $q_p(z)=0$~:

\smallskip
\footnotesize 
$\{$$N0=0; N1=0; N2=0; N3=0; H=2*10^5; B= 2*10^3; p=B; N=0.0; $\par$
while(p<B+H, p=nextprime(p+2); N=N+1; p2=p^2; Np=0; $\par$
for(z=2, p-1, Q=Mod(z,p2)^{(p-1)}-1; if(Q==0, Np=Np+1) ); $\par$
if(Np==0, N0=N0+1); if(Np==1, N1=N1+1); if(Np==2, N2=N2+1);$\par$
 if(Np>=3, N3=N3+1)); print(N0/N,"   ",exp(-1));print(N1/N,"   ",1-exp(-1) ); $\par$
print(N2/N,"   ",1-2*exp(-1));print(N3/N,"   ",1-5/2*exp(-1)) $$\}$

\smallskip
 \normalsize
Comme les probabilit\'es indiqu\'ees sont d'abord pour 0 solutions,
puis pour au moins 1 solution, 2 solutions, 3 solutions, on doit cumuler les nombres de solutions $N_1, N_2, N_3$ donn\'es par le programme
(naturellement, $N_0+N_1+N_2+N_3=N$)~:
\begin{align*}
&\hbox{cas de 0 solutions :}   &\Frac{N_0}{N} = 0.3694945; &  &\hbox{probabilit\'e} &\approx 0.3678794 \\
&\hbox{au moins 1 solution :}   &\Frac{N_1+N_2+N_3}{N} =0.6305054; &  &\hbox{probabilit\'e} &\approx  0.6321205 \\
&\hbox{au moins 2 solutions :}    &\Frac{N_2+N_3}{N} =0.2646531; &   &\hbox{probabilit\'e} &\approx 0.2642411 \\
&\hbox{au moins 3 solutions :}   &\Frac{N_3}{N} =0.0805782 ;&  &\hbox{probabilit\'e} &\approx  0.0803014 
\end{align*}
Dans ce cas, les r\'esultats num\'eriques sont remarquablement coh\'erents avec la r\'epartition probabiliste que nous allons
pr\'eciser au \S\,\ref{loi}. 

\smallskip
Noter que dans le m\^eme intervalle pour $p$, il y a 87 solutions cumul\'ees $z<\sqrt p$ pour $17866$ solutions cumul\'ees
(proportion $0.00487$). La tranche $]2.10^3, 2\,(10^3\!+\!10^5)[$ comporte 17845 nombres premiers 
(une solution en moyenne comme pr\'evu). 

\subsubsection{Commentaires au sujet des solutions ``exceptionnelles''}\label{remaconcl}

D\`es que $q_p(a) = 0$ pour $a \ll p$, plusieurs puissance de $a$ fournissent des solutions dans $[2, p[$~;
pour $p=3511$, on a les solutions $2, 4, 8, 16, 32, 64, 128, 256, 512, 1024$, $2048 < p$.
Pour $p=40487$, on a les solutions $5, 25, 125, 625, 3125, 15625 < p$; comme $4492$ est aussi une ``petite'' solution, on obtient
la solution $5\,.\,4492 = 22460 <p$, etc.

\smallskip
 La situation pr\'ec\'e\-dente
pourrait \^etre interpr\^et\'ee comme une d\'ependance de variables al\'eatoires~; cependant, en termes de solutions dans $[2, p^2[$, on trouvera toujours $p-2$ solutions $Z=z+\lambda(z)\,p$ \`a $q_p(Z)=0$, dont les pr\'ec\'edentes 
(exceptionnelles mais non suppl\'ementaires), et en un sens on peut consid\'erer qu'il ne s'agit que d'une question de r\'epartition et non d'une d\'ependance probabiliste, car alors on a ``moins de grandes solutions'' dans $[p+1, p^2[$ (par exemple, pour $p=11$ on a
$q_p(a)=0$ pour $a= 3, 9 \in [2, p[$ et $a=27, 40, 81, 94, 112, 118, 120 \,\in \  ]p, p^2[$). 

\smallskip
De fait le c\^ot\'e automatique conduisant \`a $\lambda(a^j)=0$ pour tout $a^j<p$ se rencontre pour d'autres valeurs de $\lambda(z)$~;
par exemple, pour $p=97$, on a $\lambda(z) = 41$ pour $z_1=54$, $z_2=68$, $z_3=75$, $z_4=92$~; ce ph\'enom\`ene
est d'ailleurs n\'ecessaire puisqu'on sait que beaucoup de valeurs de $q_p(z)$ ne sont pas atteintes (cf. \S\,\ref{manques}).

\smallskip
Ce type d'\'ev\'enement se produit a priori avec la m\^eme (faible) probabilit\'e, et on peut analyser ce qui pr\'ec\`ede de la fa\c con suivante~: soit $a \geq 2$ fix\'e \'etranger \`a $p$, d'ordre $d$, et soit $a_j \in [2, p[$ le r\'esidu modulo $p$ de $a^j$, $j=1,\ldots, d-1$~; posons $a^j \,a_j^{-1} \equiv 1 + \theta_j\,p \pmod{p^2}$, $\theta_j \in [0, p[$, alors on obtient 
$q_p(a_j) \equiv j\,q_p(a) + \theta_j \pmod p$.

\smallskip
Autrement dit, $j$ \'etant donn\'e, le quotient de Fermat de $a_j$ d\'epend
de celui de $a$ au moyen d'une formule canonique, le cas $q_p(a)=0$, $\theta_j=0$ pour tout $a^j < p$, n'\'etant qu'un cas 
particulier de cette formule. 

\smallskip
Le programme suivant donne la r\'epartition
des valeurs de $\lambda(z)$, $z \in [2, p[$, et celle du nombre de solutions \`a $\lambda(z)=v$, $v$ donn\'e ou pris au hasard~;
pour $10^3 \leq p \leq 10^3+10^4$ il y a $1168$ nombres premiers, et on a retenu le nombre $K$ de cas pour lesquels il y a au moins 4 solutions~:

\medskip
\footnotesize
$\{$$X=50; S=0; for(j=1,X, v=random(10^4); K=0; B=2*10^4; H=10^4;$\par$
 p=B; while(p<B+H, p=nextprime(p+2); p2=p^2; N=0;  $\par$
for(z=2, p, Q=Mod(z,p2)^p-Mod(z,p2); lambda=component(Q,2)/p;  $\par$
if(lambda==v, N=N+1) ); if(N>=4, K=K+1) );  print(v," ",K); S=S+K ) ;  $\par$
NP=0; p=B; while(p<B+H, p=nextprime(p+2); NP=NP+1);  $\par$
print(NP,"   ",S); print((S+0.0)/(X*NP))$$\}$

\smallskip
\normalsize
En prenant d'abord $v=0, \ldots, 9$, on obtient $(v, K) =$
$(0, 24)$, $(1, 21)$, $(2 , 26)$, $(3 ,17)$, $(4 , 20)$, $(5 , 33)$, $(6 , 25)$, $(7 , 21)$, $(8 , 22)$, $(9 , 21)$.

\smallskip
Pour une autre tranche de valeurs de $v$, on obtient $(v, K) =$
$(123, 21)$, $(124, 11)$, $(125 , 27)$, $(126 ,23)$, $(127  , 32)$, $(128 , 19)$, $(129  , 17)$, 
$(130 , 21)$, $(131 , 18)$, $(132 , 21)$.

\smallskip
 Dans tous les essais effectu\'es, $v=0$ ne semble pas jouer un r\^ole particulier.

\smallskip
La moyenne cumul\'ee observ\'ee pour le nombre $K$ est de $22$~; or $\frac{22}{1168} \approx 0.0188356$, et la proba\-bilit\'e 
que nous d\'efinirons pour ``au moins 4 solutions \`a $\lambda(z) =v$'' est \'egale \`a $0.0189$ (cf. Remarque \ref{rema=1}), ce qui constitue une v\'erification remarquable des arguments pr\'ec\'edents. 
Une exp\'erimentation utilisant la fonction {\it random} pour $v \in [0, 10^4[$, pour une tranche de 
$984$ nombres premiers $p>2\,.\,10^4$, conduit \`a la valeur $0.019268$.

\smallskip
Remarquons aussi que si par exemple
$q_p(2)$ \'etait nul pour une infinit\'e de $p$, alors le nombre $h$ de solutions dans $[2, p[$, d\^ues aux $a_j = 2^j$, 
tendrait vers l'infini pour une sous-suite de $p$, ce qui peut para\^itre excessif au regard de la 
r\'epartition (i.e., de la densit\'e) sur $[2, p^2[$ (cf. r\'esultats num\'eriques du \S\,\ref{num1}).

\subsubsection{Existence d'une loi de probabilit\'es}\label{loi} On suppose $z\in [2, p[$ car 1 est toujours solution.
Ce qui pr\'ec\`ede conduit \`a une heuristique 
utilisant une loi binomiale de param\`etres $\big(p-2, \Frac{1}{p} \big)$, car on peut consid\'erer que l'on r\'ealise $p-2$ tirages 
 pour lesquels on regarde combien de fois on obtient l'\'ev\'enement $\lambda(z)=0$.
 Le param\`etre $\frac{1}{p}$ est une approximation de ${\rm Prob}\big(\lambda(z)=0\big)$~; la probabilit\'e d'avoir $n$
 cas favorables exactement est $\hbox{$\binom{p-2} {n}$} \frac{1}{p^n} \big (1-\frac{1}{p} \big)^{p-2-n} = 
 \hbox{$\binom{p-2} {n}$} \frac{1}{p^{p-2}}  (p-1)^{p-2-n}$. Cette approximation pour le second
 param\`etre a une incidence n\'egligeable car $Z\in [2, p^2[$ et la probabilit\'e co\"incide avec la densit\'e.

\begin{heuristic} \label{heur4} {\it 
Soit $z\in [2, p[$ et soit $Z=z+ \lambda(z)\, p \in  [2, p^2[$ tel que $q_p(Z)=0$.
Soit $n \in [0, p-1[$~; alors la probabilit\'e d'avoir au moins $n$ valeurs $z_1, \ldots, z_n \in [2,p[$ telles que $q_p(z_j)=0$
(\'equivalent \`a $\lambda(z_j)=0$), pour $1 \leq j \leq n$, est~:
 
\centerline{ ${\rm Prob}\Big( \Big\vert \Big\{z \in [2,p[,\, q_p(z)=0 \Big\}\Big \vert  \geq n\Big) = 
\Frac{1}{p^{p-2}} \sm_{j=n}^{p-2} \hbox{$\binom{p-2} {j}$} (p-1)^{p-2-j} . $ }

\medskip
Plus g\'en\'eralement, on a pour tout $v \in [0, p[$~:

\centerline{ ${\rm Prob}\Big( \Big\vert \Big\{z \in [2,p[,\, \lambda(z) = v \Big\}\Big \vert  \geq n\Big) = 
\Frac{1}{p^{p-2}} \sm_{j=n}^{p-2} \hbox{$\binom{p-2} {j}$} (p-1)^{p-2-j} . $}  }
\end{heuristic}
 
\begin{lemma} \label{major}
On a pour tout $n$  la majoration $\Frac{1}{p^{p-2}} \sm_{j=n}^{p-2} \hbox{$\binom{p-2} {j}$} (p-1)^{p-2-j}  < 
\Frac{1}{p^n} \, \hbox{$\binom{p-2} {n}$}$.
\end{lemma}

\begin{proof}
On consid\`ere, pour $0\leq n \leq N$, $t \in [1, \infty[$, la d\'eriv\'ee de la fonction
 $f_{N,n} (t)=\sum_{j=n}^{N} \binom{N} {j} (t-1)^{N-j} - \binom{N} {n} t^{N-n}$~; elle est \'egale \`a $N f_{N-1,n} (t)$.
On raisonne ensuite par r\'ecurrence, \`a  partir de $f_{n,n} (t) = 0$ et de $f_{N,n} (1) <0$, pour montrer que la 
d\'eriv\'ee est n\'egative ou nulle sur tout l'intervalle $[1, \infty[$. On aura ensuite \`a poser $t=p$, $N=p-2$. 
\end{proof}

\begin{remark}\label{rema=1} On a, pour les petites valeur de $n$, la formule plus commode~:
$${\rm Prob}\Big(\Big\vert \Big\{z \in [2,p[,\, q_p(z)=0 \Big\} \Big\vert  \geq n\Big) = 
1 - \sm_{j=0}^{n-1} \hbox{$\binom{p-2} {j}$}\Frac{1}{p^j}\Big (1-\Frac{1}{p} \Big)^{p-2-j} , $$
et de m\^eme pour la condition $\lambda(z) = v$ \`a la place de $q_p(z)=0$ (cas $v = 0$).

\smallskip
La probabilit\'e d'avoir au moins une solution $z \in [2, p[$ est donc $1 - \Big(1-\Frac{1}{p} \Big)^p \Big(\Frac{p}{p-1}\Big)^2$
qui est rapidement proche de $1- e^{-1}\,\Big(\Frac{p}{p-1}\Big)^2$ donc de $1-e^{-1} \approx 0.63212$. Pour au moins 2 solutions
on obtient une probabilit\'e proche de $1- 2 \, e^{-1}\,\Big(\Frac{p}{p-1} \Big)^2\approx 0.264$~; pour au moins 3 (resp. 4) solutions
on obtient $0.0803$ (resp. $0.0189$).

\smallskip
La probabilit\'e d'avoir 0 solutions est donc $\Big(1-\Frac{1}{p} \Big)^p \Big(\Frac{p}{p-1} \Big)^2\approx 0.3678$. 
L'excel\-lence des r\'esultats num\'eriques accr\'edite l'existence d'une loi de probabilit\'e~binomiale.
\end{remark}

Pour $a\ll p$, ${\rm Prob}\big(q_p(a)=0\big)$ est conditionn\'ee \`a ${\rm Prob}\big(n \geq h \big)$, o\`u $h$ est la
partie enti\`ere de $\Frac{{\rm log}(p)}{{\rm log}(a)}$ (cf. \S\,\ref{hp})~; or le rapport 
$\Frac{{\rm Prob}\Big( \Big\vert \Big\{z \in [2,p[,\, q_p(z) =0 \Big\}\Big \vert  \geq h\Big) } {p^{-h} \, \hbox{$\binom{p-2} {h}$}} < 1$
tend vers une constante $C_\infty(a)$ en d\'ecroissant selon le r\'esultat suivant~:

\begin{lemma} \label{encadre}
On a pour tout $p$ l'encadrement (cf. Lemme \ref{major})~:

\medskip
\centerline{${\rm exp}\big(-1 + \frac{1}{p} (h + \frac{3}{2} )\,\big)<
\Frac{p^{-(p-2)} \sum_{j=h}^{p-2} \hbox{$\binom{p-2} {j}$} (p-1)^{p-2-j}  } {p^{-h} \, \hbox{$\binom{p-2} {h}$} } \leq 1$.}
\end{lemma}

\begin{proof} On a la minoration $\ \ \Frac {p^h} {\hbox{$\binom{p-2} {h}$} }
\times \Frac{1} {p^{p-2}}\sm_{j=h}^{p-2} \hbox{$\binom{p-2} {j}$} (p-1)^{p-2-j}$
\begin{align*}
 = \ &  \Big(\Frac{p-1}{p}\Big)^{p-2}\,
 \Frac {p^h \,h!}{(p-1-h) \cdots (p-1-1)} \ \sm_{j=h}^{p-2} \ \Frac{1}{j!} \,\Frac{p-1-j}{p-1} \cdots \Frac{ p-1-1}{p-1 } \\
  =\  &\Big(\Frac{p-1}{p}\Big)^{p-2}\,
  \Frac {p^h}{(p-1)^h} \ \sm_{j=h}^{p-2} \ \Frac{ h!}{j!} \,\Frac{p-1-j}{p-1-h} \cdots \Frac{ p-1-1}{p-1-1 } \times \Frac{1}{(p-1)^{j-h}}\\
 =\  &\Big(\Frac{p-1}{p}\Big)^{p-2-h}\,\Big[ 1+ \Frac{p-1 -(h+1)}{(p-1)(h+1)} + \cdots +\Frac{p-1 -(h+1)}{(p-1)(h+1)}  \, \cdots \, \Frac{p-1 -j}{(p-1)\,j}  \\
&\hspace{6cm} 
 + \cdots  + \Frac{p-1 -(h+1)}{(p-1)(h+1)}  \, \cdots \, \Frac{p-1 -(p-2)}{(p-1)(p-2)} \Big]\\
 > \ &\Big(\Frac{p-1}{p}\Big)^{p-2-h} =  \Big(1 -\Frac{1}{p}\Big)^{p-2-h}. \  \hbox{D'o\`u facilement le r\'esultat en consid\'erant~:}
\end{align*}

$(p-2-h) \,{\rm log}\Big(1-\Frac{1}{p}\Big) = -(p-2-h) \,\Big(\Frac{1}{p} + \Frac{1}{2p^2}+ \cdots\Big )
> -1 + \Frac{1}{p}\Big(h+ \Frac{3}{2} \Big)$,
tous les termes n\'eglig\'es \'etant positifs et tendant rapidement vers 0.
\end{proof}

La constante $C_\infty(a)$ est voisine de $e^{-1} \approx 0.36788$, et pour $p \to \infty$ on peut \'ecrire~:
$${\rm Prob}\Big(\Big\vert \Big\{z \in [2,p[,\, q_p(z)=0\, \Big\} \Big\vert  \geq h\Big) \approx C_\infty(a) \times 
\Frac{1}{p^h} \, \hbox{$\binom{p-2} {h}$} \approx O \Big( \Frac{1}{p^{ {\rm log}_2(p)/ {\rm log}(a)}} \Big), $$

ordre de grandeur  qui sera obtenu au niveau de la preuve du Th\'eor\`eme \ref{theoremconv}.

\smallskip
Par exemple, pour $a=2$, $p = 100000007$, on obtient un rapport (effectivement majorant) de $0.3820$ au lieu de $0.36788$. 
Pour $p = 100003$  on obtient $0.3908$.
On a utilis\'e le programme suivant~:

\smallskip
\footnotesize 
$\{$$a=2; p=nextprime(10^3); print(p); h=floor(log(p)/log(a)); S=0.0; $\par$
for(k=1, p-2-h, S=(S+1)*k/((p-1)*(p-1-k))  ); S=S+1; print(exp(-1)*S)$$\}$
\normalsize

\begin{example}\label{ex4}
Donnons, sous les heuristiques pr\'ec\'edentes, des calculs exacts de probabilit\'es d'avoir au moins $h$
solutions, o\`u $h$ est la partie enti\`ere de $\frac{{\rm log}(p)}{{\rm log}(a)}$ (ici avec $a=2$) et o\`u $p$ est arbitrairement grand~; 
ceci correspondrait au cas o\`u le quotient de Fermat de $a$ serait nul pour une infinit\'e de $p$ et il convient de voir que c'est 
num\'eriquement incompatible. 
On \'ecrit alors cette probabilit\'e sous la forme~$\frac{1}{p^{1+\epsilon}}$~:

\medskip
\footnotesize 
$\{$$p=nextprime(10^6); S=0.0; $\par$
 for(j=0, log(p)/log(2), S=S+binomial(p-2, j)*(1-1/p)^{(p-2-j)} / p^j  ); $\par$
 print(p,"   ",1-S,"   ",-1-log(1-S)/log(p)) $$\}$  
\begin{align*}
&p=101    &\hbox{probabilit\'e} &= 6.269 \times 10^{-5} &     & \epsilon  = 1.097\\
&p=127   &\hbox{probabilit\'e} &= 6.655 \times 10^{-5} &     & \epsilon  = 0.985 \\
&p=10007   &\hbox{probabilit\'e} &= 4.473 \times 10^{-12} &     & \epsilon  = 1.837\\
&p=200003   &\hbox{probabilit\'e} &= 6.059\times 10^{-17} &      & \epsilon  =  2.059\\
&p=1000003 &\hbox{probabilit\'e} &= 1.587 \times 10^{-19} &   & \epsilon  =   2.133
\end{align*}
\end{example}

\normalsize
On confirmera dans la section suivante que cette probabilit\'e est rapidement inf\'e\-rieure \`a $\frac{1}{p^2}$
et m\^eme que $\epsilon$ tend vers l'infini tr\`es lentement. Pour les petites valeurs de $p$, $\epsilon$ oscille autour de 1
et la derni\`ere valeur de $p$ pour laquelle $\epsilon<1$ est $p=127$.

\subsection{Heuristique principale sur $q_p(a)=0$}\label{hp}
Soit maintenant $a\geq 2$ fix\'e. 
 L'\'ev\'ene\-ment $q_p(a)=0$ (o\`u $p$ assez grand est la variable al\'eatoire) 
est \'equivalent au suivant, o\`u $h \geq 1$ est la partie enti\`ere de $\Frac{{\rm log}(p)}{{\rm log}(a)}$~:

\smallskip
{\it Il existe au moins $h$ entiers $z_1, \ldots, z_h$ de $[2,p[$ tels que  $\lambda(z_j)=0$ (i.e., $q_p(z_j)=0$) pour $j=1, \ldots, h$,
et il existe un indice $j_0$ tel que $z_{j_0} = a$.  }

\medskip
Si $q_p(a)=0$, l'existence des $h$ \'el\'ements $z_j \in [2, p[$  tels que  $\lambda(z_j)=0$ avec $z_{j_0} = a$
en r\'esulte trivialement ($z_j=a^j \in [2, p[$ pour $j=1,\ldots, h$).
Inversement, sous l'exis\-tence de $h$ \'el\'ements  $z_j$ tels que  $\lambda(z_j)=0$, la seule condition 
$\{z_1, \ldots, z_h \}$ contient $a$ entra\^ine  $q_p(a)=0$.

\begin{remark}\label{rema22} 
 L'existence de $n$ valeurs $z_j \in [2, p[$ telles que $q_p(z_j)=0$ ne d\'epend que de $p$ (et de $n$) et non du choix
d'un entier $a$ (fait a posteriori). Ceci dit, il y a de fortes chances que ce soit d\^u \`a l'existence d'un $a \ll p$
tel que $q_p(a)=0$. Cette derni\`ere probabilit\'e ($\{z_1, \ldots, z_h \}$ contient $a$) est difficile \`a estimer, 
aussi nous la majorerons par 1 (si $a \notin \{z_1, \ldots, z_h\}$, on obtient plus de $h$ solutions, ce qui est peu probable).

\smallskip
 Il est clair que les $p$ pour lesquels le nombre $n$ de solutions dans $[2, p[$ est tr\`es petit conduisent \`a $q_p(b)\ne 0$ 
pour tout $b<p^{\frac{1}{n+1}}$, $b\ne 1$ (cf. \S\,\ref{subdens}).

Le cas de $h$ solutions donn\'ees par les puissances de $a$ peut \^etre consid\'er\'e comme un cas tr\`es particulier (probabilit\'e conditionnelle) du cas de $h$ solutions ind\'ependantes dont la probabilit\'e reste 
$\Frac{1}{p^{p-2}} \sm_{j=h}^{p-2} \hbox{$\binom{p-2} {j}$} (p-1)^{p-2-j}$.
On obtient alors dans ce contexte (cf. Heuristique \ref{heur4})
${\rm Prob}\big(q_p(a) = 0\big) < \frac{\binom{p-2} {h}}{p^h}$ et
${\rm Prob}\big(q_p(a) = 0\big) \approx C_\infty(a) \times \frac{\binom{p-2} {h}}{p^h}$ (cf. Lemme \ref{encadre}).
Pour $p<a$, $h=0$, et $\frac{\binom{p-2} {h}}{p^h}=1$~; donc
il est pr\'ef\'erable, dans l'optique de l'\'etude de la sommation sur $p$, d'utiliser la densit\'e 
$\sm_{d \div p-1} \Frac{\varphi(d)^2}{p\,(p-1)^2}$ \'etudi\'ee Section \ref{section3}.
\end{remark}

\begin{theorem}\label{theoremconv} Soit $a \geq 2$.
 La s\'erie $\sm_{p\geq 2}\  \Frac{\binom{p-2} {h}}{p^h}$, o\`u $h$ est la partie enti\`ere de 
$\Frac{{\rm log}(p)}{{\rm log}(a)}$, est convergente.
\end{theorem}

\begin{proof} On a $\binom{p-2} {h} = \Frac{1}{h !} \times (p-1-1)  \cdots (p-1-h)$ 
que l'on peut majorer par $\Frac{1}{h !} \times p^h$. En outre, on a par d\'efinition
$\frac{{\rm log}(p)}{{\rm log}(a)} -1 < h < \frac{{\rm log}(p)}{{\rm log}(a)}$.
 Pour tenir compte de ce fait et afin d'utiliser analytiquement $\frac{{\rm log}(p)}{{\rm log}(a)}$
 au lieu de $h$ dans les formules, on utilise la majoration 
$\sm_{p\geq 2}\  \Frac{\binom{p-2} {h}}{p^h} < \sm_{p\geq 2} \Frac{h}{h !}$,
o\`u l'on a remplac\'e $\Frac{1}{h!}$ par le majorant 
$1\big/  \big(\frac{{\rm log}(p)}{{\rm log}(a)}-1\big)! =
\frac{{\rm log}(p)}{{\rm log}(a)} \big/  \big(\frac{{\rm log}(p)}{{\rm log}(a)}\big)!$,
$h$ d\'esignant maintenant $\frac{{\rm log}(p)}{{\rm log}(a)}$~; d'o\`u 
$\Frac{h}{h !}= \frac{1}{\Gamma(h)}$.

\smallskip
On a $h! =h\,\Gamma(h)= \sqrt{2\pi h} \times h^h e^{-h}\times (1 +O(\frac{1}{h}))$ et 
$\Frac{h !}{h} =  \sqrt{2\pi}\times h^{h-\frac{1}{2}} \, e^{-h}\times (1 +O(\frac{1}{h}))$. 
\begin{align*} \hbox{Or~:\ }
{\rm log}\Big(\Frac{h!}{h}\Big) &= {\rm log}\big (\sqrt{2\pi} \big) +\Big (h-\Frac{1}{2}\Big) {\rm log}(h) -h 
+ {\rm log}\Big(1 +O\big(\Frac{1}{h}\big)\Big) \\
&=  {\rm log}(\sqrt{2\pi}) + h ( {\rm log}(h) -1) - \Frac{1}{2}  {\rm log}(h) + O\big(\Frac{1}{h} \big) \\
& =  {\rm log}(\sqrt{2\pi}) +\Frac{1}{{\rm log}(a)}  {\rm log}(p)\Big (  {\rm log}_2(p) -  {\rm log}_2(a) -1 \Big) \\
&\hspace{1.0cm}  - \Frac{1}{2} \Big( {\rm log}_2(p) -  {\rm log}_2(a) \Big) + O\big(\Frac{1}{ {\rm log}(p)} \big)  \\
&=  \Big[ \Frac{1}{ {\rm log}(a)} \Big ( {\rm log}_2(p) -  {\rm log}_2(a) -1\Big)   \\
& \hspace{1.0cm}  - \Frac{1}{2} \Frac{1}{ {\rm log}(p)}\Big ({\rm log}_2(p) - 
 {\rm log}_2(a) \Big) + \Frac{O(1)}{ {\rm log}(p)}  \Big]\,  {\rm log}(p)  =: Y\times  {\rm log}(p)\,.
\end{align*}
D'o\`u $\Frac{h}{h!} = \Frac{1}{p^Y}$ o\`u $Y$ tend vers l'infini comme $\Frac{{\rm log}_2(p)}{ {\rm log}(a)}$.
Par cons\'equent, il existe une constante $C>1$ telle que $Y$ est minor\'ee par $C$ pour tout $p\geq p_0$ assez grand
et on peut \'ecrire $\sm_{p\geq 2}\  \Frac{\binom{p-2} {h}}{p^h} < C_0 + \sm_{p > p_0}\Frac{1}{p^C}$,
o\`u $C_0$ est une constante \'egale \`a la sommation partielle jusqu'\`a $p_0$~;
d'o\`u la convergence de la s\'erie intiale.
\end{proof}

\begin{heuristic}\label{heur11}    {\it 
Soit $a \geq 2$ fix\'e~; alors on a ${\rm Prob}\big(q_p(a) = 0\big) \approx C_\infty(a) \times \frac{\binom{p-2} {h}}{p^h}$,
o\`u $C_\infty \approx 0.36788$, $h$ est la partie enti\`ere de $\Frac{{\rm log}(p)}{{\rm log}(a)}$,
et  dans le cadre du principe de Borel--Cantelli,  le nombre de $p$ tels que $q_p(a) = 0$ est major\'e par la limite 
de la s\'erie  $S :=s_0 + \sm_{p>a} \ \Frac{\binom{p-2} {h}}{p^h}$,
o\`u $s_0 \approx \sm_{p<a}\ \sm_{\ d \div p-1} \Frac{\varphi(d)^2}{p\,(p-1)^2} < \sm_{p<a} \Frac{1}{p}$. }
\end{heuristic}

Noter que la majoration utilis\'ee pour le Th\'eor\`eme \ref{theoremconv} est assez grossi\`ere car la s\'erie 
$\sum_{p\geq 2}\frac{1}{p^h}\binom{p-2}{h}$
converge vers $0.9578...$ (pour $a=2$) tandis que
$\sum_{p\geq 2}\Frac{h}{h!}$ converge vers $6.2761..$. Par cons\'equent, la s\'erie de d\'epart 
$\sum_{p\geq 2}\Frac{1}{p^{p-2}} \sum_{j=h}^{p-2} \hbox{$\binom{p-2} {j}$} (p-1)^{p-2-j}$ converge vers 
$C_\infty(2)\times 0.9578... \approx 0.35237$.
Ces  constantes augmentent rapidement avec $a$.

\smallskip
 Le fait que l'on puisse choisir $C$ arbitrairement grande (\`a condition de sommer \`a
partir d'un $p_0$ assez grand) montrerait la rar\'efaction des solutions pour $p \to \infty$. 

\smallskip
Par exemple, si $a=2$ et si l'on veut
atteindre $C>1$, il faut avoir $p_0 \geq 79$~; pour $C>2$, il faut $p_0 \geq 4259$. Pour $a=3$, il faut respectivement
$p_0 \geq 24527$ et $p_0 \geq 2669180065451$. Pour $a=5$ et $C \approx 1.05$, $p_0=168116638259$
(peut-on y voir un rapport avec l'exemple $(5, 188748146801)$ donn\'e au \S\,\ref{ex3} ?).

\smallskip
Ces r\'esultats sont obtenus avec le programme suivant qui concerne la s\'erie majorante 
$\sm_{p\geq p_0} \Frac{h}{h !} \approx \sm_{p\geq p_0}\Frac{1}{p^Y}$, donc
les $p_0$ obtenus sont des majorants des bornes n\'eces\-saires pour avoir une s\'erie initiale convergente
comme celle de terme g\'en\'eral~$\Frac{1}{p^Y}$~:

\smallskip\footnotesize 
$\{$$a=5; print(nextprime(solve(x=10^2, 10^{12},  $\par$
 (log(log(x))-log(log(a))-1)/log(a)- (log(log(x))-log(log(a)))/log(x^2) - 1.05   )))$$\}$

\medskip\normalsize
Cette heuristique \ref{heur11} donne une version sans doute trop favorable du probl\`eme, mais elle est assez bien
v\'erifi\'ee par l'exp\'erimentation num\'erique. Le paragraphe suivant, qui utilise des r\'esultats de densit\'es,
peut pr\'eciser cet aspect.

\subsection{Etude \`a l'infini}\label{subinfini}
D'apr\`es les r\'esultats des \S\S\,\ref{subtild}, \ref{subdfp}, pour $a$ fix\'e
on est amen\'e \`a \'etudier le produit infini formel $\widetilde{\mathcal P}(a) := \prod_{m \geq 1} \widetilde\Phi_m(a)$ qui est tel que tout nombre premier $p\notdiv a$ en est un diviseur, \`a savoir $p\div \widetilde\Phi_m(a)$ pour l'unique indice $m=o_p(a)$ 
(cf. Lemmes \ref{dec1}, \ref{dec2}),
et qui est tel que $q_p(a) \ne 0$ si et seulement si $p^2$ ne divise pas $\widetilde{\mathcal P}(a)$.
Pour \'etudier les $q_p(A)$ non nuls en termes de densit\'es, on va consid\'erer les densit\'es des $A\in \N$ tels que 
$p^2 \notdiv \widetilde {\mathcal P}(A)$ (cf. Section~\ref{section3}).

\smallskip
Comme $p \div \widetilde {\mathcal P}(A)$ est \'equivalent \`a $p\div \widetilde \Phi_{o_p(A)}(A)$,
la densit\'e des $A\in \N$ tels que $p^2 \div \widetilde {\mathcal P}(A)$
est \'egale \`a $\Frac{\varphi(o_p(A))}{p^2}$ et en sommant sur tous les ordres possibles $o_p(A)$ diviseurs de $p-1$, on obtient la densit\'e
$\Frac{p-1}{p^2}$~; la densit\'e contraire ($p^2 \notdiv \widetilde {\mathcal P}(A)$) est \'egale \`a
$D_p := 1 -\Frac{p-1}{p^2} = 1- \Frac{1}{p}+ \Frac{1}{p^2}$. On note que ces $p$-densit\'es sont ind\'ependantes
(en raison des propri\'et\'es des $\widetilde\Phi_m(a)$) et que la densit\'e
correspondant \`a plusieurs $p$ est donn\'ee par le produit des densit\'es locales (voir ci-dessous la Remarque \ref{remadens}).

\smallskip
Il convient d'\'etudier le produit $\prd_{p \leq x} D_{p}$ qui donne la densit\'e des $A\in \N$ tels que 
$p^2  \notdiv \widetilde{\mathcal P}(A)$ pour tout $p \leq x$.
Noter que seules les valeurs de $m$ de la forme $o_p(A)$, pour un $p \leq x$, sont concern\'ees dans le produit infini.

\smallskip
Ecrivons $1- \Frac{1}{p}+ \Frac{1}{p^2}  = \Big (1- \Frac{1}{p}\Big)\Big (1+ \Frac{1}{p(p-1)}\Big)$.
On a~:
$$\prd_{p \leq x}  \Big  (1 -\Frac{1}{p}\Big ) = 
\Frac{e^{-\gamma}}{{\rm log}(x) } \times \Big( 1 + O\big (\hbox{$\frac{1}{{\rm log}(x)}$} \big)\Big), $$
o\`u $\gamma \approx 0,577215$ est la constante d'Euler (cf.  \cite{T1}, \S\,I.1.6, formule de Mertens), et
$$\prd_{p \leq x}  \Big  (1+ \Frac{1}{p(p-1)}\Big ) \approx 1.9436, $$
 d'o\`u~:
$$\prd_{p \leq x} D_{p} \approx \Frac{1.9436 \times e^{-\gamma}}{{\rm log}(x) } \times
 \Big( 1 + O\big (\hbox{$\frac{1}{{\rm log}(x)}$} \big)\Big) \approx \Frac{1.09125} {{\rm log}(x) } \times
 \Big( 1 + O\big (\hbox{$\frac{1}{{\rm log}(x)}$} \big)\Big). $$

 On a donc le r\'esultat analytique suivant~:

\begin{theorem} \label{theo2}  La densit\'e des $A \in \N \Sauf \{0\}$ satisfaisant aux propri\'et\'es locales~: 
``$q_p(A) \ne 0$ pour tout premier $p \leq x$'', est de l'ordre de $\Frac{O(1)}{{\rm log}(x)}$.
De fa\c con pr\'ecise~:

\medskip
\centerline{$\lim\limits_{\substack{\ \ y \to \infty}} \Frac{1}{y}\,\Big\vert \big\{  A \leq y, \ \ q_p(A)\ne 0 ,\ \forall p \leq x \big\} \Big\vert 
= \Frac{O(1)}{{\rm log}(x)} \approx  \Frac{1.09125}{{\rm log}(x)}. $}
\end{theorem}

\smallskip
\begin{remark}\label{remadens} De fait il existe un calcul direct de cette densit\'e par d\'enombre\-ment
de type th\'eor\`eme chinois (cf. Remarque \ref{rema11}) avec cette fois des $B_p^j$ tels que $q_p(B_p^j) \ne 0$, 
et ceci pour la suite des nombres premiers $p\leq x$.
Si $y=\prod_{p\leq x} p^2$, un calcul standard montre que le nombre de $A \in [1, y[$ tels que $q_p(A) \ne 0$ 
pour tout $p \leq x$ est exactement $\prod_{p \leq x} (p^2 - p + 1)$, en notant que $A$ est par nature
non \'etranger \`a $\prod_{p \leq x} p$~; d'o\`u la densit\'e pr\'ec\'edente exacte sur les intervalles de la forme
$\big [1,\prod_{p\leq x} p^2 \big [$.
Ceci constitue une importante v\'erification des r\'esultats de la Section \ref{section3} et montre que la conjecture $ABC$
n'est pas n\'ecessaire dans ce cadre cyclotomique.
\end{remark}

Bien que $y$ doive \^etre pris tr\`es grand par rapport \`a $x$, on peut tester la r\'epartition des solutions sur de petits intervalles
en utilisant le programme suivant~:

\medskip
\footnotesize 
$\{$$N=0; y=10^4;  x=10^7; A=1; while(A<=y, A=A+1; p=0; q = 1; $\par$
while(p<=x \& q!=0, p=nextprime(p+1); p2=p^2; $\par$
Q=Mod(A,p2)^{(p-1)}-1; q=component(Q,2)) ; if(q!=0, N=N+1));  print(N)$$\}$

\medskip
\normalsize
Par exemple, pour $1 <A \leq y=10^4$, on trouve $665$ valeurs de $A$ telles que $q_p(A) \ne 0$ 
pour tout  $p \leq x=10^7$. Or $10^4\,. \,\Frac{1.09}{{\rm log}(10^7)} \approx 676$.

Du fait que le programme compte les plus petites solutions $A$ \`a $q_p(A) \ne 0$ pour tout $p\leq x$, 
sans doute moins nombreuses\,\footnote{\, La relation
$q_p(a)=0$ engendre les solutions $a^j \in [2, p[$, $j=1,\ldots, h$, qualifi\'ees 
d'exception\-nelles (cf. \S\,\ref{num2}), et qui sont ici d\'ecompt\'ees des $A$ telles que $q_p(A)\ne 0 ,\ \forall p \leq x$.},
le r\'esultat est assez satisfaisant.
Prenons $x \approx 10^{10}$, accessible aux calculs~; on a $\Frac{1.09}{{\rm log}(10^{10})} \approx 0.05$.
Pour les entiers $A \in \N \Sauf \{0, 1\}$, il y en a $95\%$ tels que $q_p(A)=0$ pour au moins un $p\leq 10^{10}$. 
Ceci est compatible avec une heuristique de finitude~; les exemples de $a=47$ et $72$ semblent \^etre int\'eressants 
de ce point de vue (cf. \S\,\ref{ex3}).

\smallskip
Cette \'etude est de type ``densit\'e'' et n'informe que tr\`es partiellement sur le cas d'une valeur $a$ fix\'ee une fois pour toutes.

\subsection{Heuristique de finitude}
On peut enfin envisager l'heuristique assez radicale suivante, en tenant compte des r\'esultats du \S\,\ref{hp}~: 

\begin{heuristic} \label{heur3} {\it 
 Soit $a \in\N \Sauf\{0, 1\}$ un entier fix\'e. Le nombre de quotients de Fermat $q_p(a)$ nuls
est en moyenne \'egal \`a 2 ou 3. }
\end{heuristic}

Le programme suivant donne $2.76$ solutions $p < 3 \times 10^9$ en moyenne pour $2 \leq a \leq 101$,
et $2.80$ solutions $p < 10^9$ pour $10^9+1 \leq a \leq 10^9 + 100$~:

\medskip
\footnotesize 
$\{$$N = 0; b=1; B= 10^8; for(a=b+1, b+100, if(Mod(a,4)==1, N=N+1)); $\par$
p= 1; while(p < B, p=nextprime(p+2); p2=p^2; $\par$
for(a=b+1, b+100, Q=Mod(a,p2)^{(p-1)} ; if(Q==1, N=N+1) ) ); print(N/100.0)$$\}$

\medskip
 \normalsize
Le fait de cumuler une centaine de valeurs de $a$ semble indispensable au vu de la r\'epar\-tition tr\`es incertaine
des solutions $p$ \`a $q_p(a)=0$ pour un seul $a$.

\smallskip
L'exp\'erimentation num\'erique (en d\'epit du fait que l'on a des ph\'enom\`enes qui exi\-gent des intervalles \`a 
croissance exponentielle) semble limiter le nombre de $q_p(a)=0$ \`a quelques unit\'es en moyenne
portant en premier lieu sur de petits $p$ (r\'esultant de congruences du type $a \equiv 1 \pmod {p^2}$)
puis \'eventuellement sur un petit nombre de grandes solutions, accessibles aux ordinateurs actuels, dont la probabilit\'e 
serait de l'ordre de $\frac{1}{p^2}$ et tendrait rapidement vers 0 pour les tr\`es grands nombres premiers comme 
l'heuristique principale semble l'indiquer (cf. Heuristique \ref{heur11}, Th\'eor\`eme \ref{theoremconv}).

\section{Conclusion}\label{section5}
N'\'etant pas familier de la th\'eorie analytique des nombres, j'ignore si l'on peut envisager des confirmations ou infirmations 
des heuristiques propos\'ees. 

\smallskip
L'Heuristique \ref{heur2} est probablement tr\`es raisonnable, mais est insuffisante pour conclure \`a la finitude des $p$
tels que $q_p(a)=0$ ($a$ fix\'e). Si elle est exacte, elle montre que la probabilit\'e $\frac{1}{p}$, souvent admise, pose probl\`eme.

\smallskip
 L'Heuristique \ref{heur4}, qui stipule l'existence
d'une loi de probabilit\'e binomiale pour ${\rm Prob}\big( q_p(z)=0 \big)$, $z \in [2, p[$, reste le point sensible en raison
de l'existence possible de nombres $a \ll p$ tels que $q_p(a^j)=0$ pour $j=1, \ldots, h$, o\`u $h$ est la partie enti\`ere de
$\frac{{\rm log}(p)}{{\rm log}(a)}$. Dans ce cas, l'abondance de solutions (car $a^j \in [2, p[$ pour $j = 1, \ldots, h$) induit 
une r\'epartition exeptionnelle des solutions qui peut \^etre interpr\^et\'ee de deux fa\c cons~: ou bien cette loi de probabilit\'e 
n'est pas la bonne, ou bien il n'est pas possible que pour $a$ fix\'e ($a=2$ par exemple) on ait une infinit\'e de solutions $p$ 
\`a $q_p(a)=0$ car alors pour ces premiers $p$ le nombre de solutions $a_i \in[2, p[$ cro\^it comme $O(1) {\rm log}(p)$,
ce qui peut appara\^itre comme une proportion excessive.

\smallskip
Ceci dit, l'\'etude pr\'ec\'edente, quoique tr\`es insuffisante, ainsi que les exp\'erimentations num\'eriques, me confortent dans la validit\'e des conjectures que j'ai formul\'ees dans le cadre tr\`es g\'en\'eral des r\'egulateurs $p$-adiques d'un nombre alg\'ebrique $\eta$ (cas Galoisien arbitraire) pour lesquels le quotient de Fermat n'est autre que le cas particulier de la $\theta$-composante, pour le caract\`ere unit\'e 
$\theta=1$, du r\'egulateur de $\eta$ (cf. \cite{Gr}).

\section{Remerciements}
Je remercie G\'erald Tenenbaum pour ses indications de th\'eorie analytique des nombres, dont
sa contribution \cite{T2}, et pour sa disponibilit\'e.

\end{document}